\documentclass[10pt,reqno]{amsart}

\usepackage{setspace}
\setdisplayskipstretch{1.5}
\usepackage[T1]{fontenc}
\usepackage{lmodern}

\usepackage{mathtools}
\usepackage{amsmath,amsfonts,amsbsy,amsgen,amscd,mathrsfs,amssymb,amsthm}
\usepackage{enumerate}
\usepackage{bm}
\usepackage{calc}

\usepackage{stmaryrd}
\SetSymbolFont{stmry}{bold}{U}{stmry}{m}{n} 

\usepackage[usenames,dvipsnames]{xcolor}
\usepackage[colorlinks=true,citecolor=blue,linkcolor=blue]{hyperref}

\usepackage{tikz}
\usepackage[font=small, margin=.5cm]{caption}

\newtheorem{thm}{Theorem}[section]
\newtheorem{lem}[thm]{Lemma}
\newtheorem{conj}[thm]{Conjecture}
\newtheorem{prop}[thm]{Proposition}
\newtheorem{cor}[thm]{Corollary}
\newtheorem{fact}[thm]{Fact}

\theoremstyle{definition}

\newtheorem{defn}[thm]{Definition}

\newtheorem{prb}[thm]{Problem}
\newtheorem{example}[thm]{Example}

\theoremstyle{remark}

\newtheorem*{qst*}{Question}
\newtheorem{rem}[thm]{Remark}
\newtheorem*{rem*}{Remark}
\newtheorem*{claim}{Claim}

\numberwithin{equation}{section} 
\numberwithin{figure}{section}
\numberwithin{table}{section}

\makeatletter

\newcommand{\supp}{\mathop{\mathrm{supp}}}
\newcommand{\intr}{\mathop{\mathrm{int}}}
\newcommand{\cl}{\mathop{\mathrm{cl}}}
\newcommand{\V}{\mathsf{V}}
\newcommand{\D}{\mathsf{D}}
\newcommand{\proj}{\mathrm{P}}
\newcommand{\relint}{\mathop{\mathrm{relint}}}
\newcommand{\relbd}{\mathop{\mathrm{relbd}}}
\newcommand{\bd}{\mathop{\mathrm{bd}}}
\newcommand{\im}{\mathop{\mathrm{Im}}}
\newcommand{\graff}{\mathrm{Graff}}
\newcommand{\SM}{\mathrm{S}}
\newcommand{\HH}{\mathrm{H}}

\begin{document}

\title{On Minkowski's monotonicity problem}

\author[Van Handel]{Ramon van Handel}
\address{Department of Mathematics, Princeton University, 
Princeton, NJ 08544, USA}
\email{rvan@math.princeton.edu}

\author[Wang]{Shouda Wang}
\address{PACM, Princeton University, 
Princeton, NJ 08544, USA}
\email{shoudawang@princeton.edu}

\begin{abstract}
We address an old open question in convex geometry that dates back
to the work of Minkowski: what are the equality cases of the monotonicity
of mixed volumes? The problem is equivalent to that of providing
a geometric characterization of the support of mixed area measures.
A conjectural characterization was put forward by Schneider (1985), but 
has been verified to date only for special classes of convex bodies.
In this paper we resolve one direction of Schneider's conjecture for 
arbitrary convex bodies in $\mathbb{R}^n$, and resolve the full conjecture 
in $\mathbb{R}^3$. 
Among the implications of these results is a mixed counterpart of the
classical fact, due to Monge, Hartman--Nirenberg, and Pogorelov, that 
a surface with vanishing Gaussian curvature is a ruled surface.
\end{abstract}



\maketitle

\section{Introduction and main results}
\label{sec:intro}

The foundation for the modern theory of convex geometry was laid by 
Minkowski in a seminal 1903 paper \cite{Min03} and in a longer manuscript 
from around the same time that was published posthumously \cite{Min11}. A 
number of basic questions that were raised in these works remain open to 
this day. The aim of the present paper is to significantly advance what is 
known about one of these problems: the characterization of the equality 
cases of the monotonicity of mixed volumes. We also discuss implications 
to the related notion of mixed Hessian measures, and to mixed analogues of 
the solution of homogeneous Monge-Amp\`ere equations.

\subsection{Mixed volumes and the monotonicity problem}

For any convex bodies $K_1,\ldots,K_m$ in $\mathbb{R}^n$ and
$\lambda_1,\ldots,\lambda_m\ge 0$, the volume 
$$
	\mathrm{Vol}_n(\lambda_1K_1+\cdots+\lambda_mK_m)
	=
	\sum_{i_1,\ldots,i_n=1}^m
	\V_n(K_{i_1},\ldots,K_{i_n})\,
	\lambda_{i_1}\cdots\lambda_{i_n}
$$
is a homogeneous polynomial. Its coefficients $\V_n(C_1,\ldots,C_n)$,
called \emph{mixed volumes}, are important geometric parameters that 
capture many familiar quantities (volume, surface area, mean width, 
projection volumes, $\ldots$) as special cases \cite{BF87,Sch14}.

One of the simplest properties of mixed volumes is their monotonicity:
if $K\subseteq L$ and $C_1,\ldots,C_{n-1}$ are convex bodies in 
$\mathbb{R}^n$, then
\begin{equation}
\label{eq:mono}
	\V_n(K,C_1,\ldots,C_{n-1}) \le \V_n(L,C_1,\ldots,C_{n-1}).
\end{equation}
This paper is concerned with the following
open problem of Minkowski \cite[\S 28]{Min11}.

\begin{prb}
\label{prb:mink}
When does equality hold in \eqref{eq:mono}?
\end{prb}

The analogous question for volume is trivial: $K\subseteq L$ and 
$\mathrm{Vol}(K)=\mathrm{Vol}(L)>0$ imply $K=L$. In contrast, 
\eqref{eq:mono} has a rich family of equality cases that gives rise to 
surprising phenomena. Let us illustrate this with an example in 
$\mathbb{R}^3$. Thoughout this paper, $B$ always denotes the Euclidean 
unit ball.

\begin{example}
For any convex body $K$ in $\mathbb{R}^3$,
the surface area and mean width can be expressed as
$\V_3(K,K,B) = \frac{1}{3}\mathrm{S}(K)$ and $\V_3(K,B,B) = \frac{2\pi}{3} 
\mathrm{W}(K)$, respectively.
Thus \eqref{eq:mono} gives rise to the geometric inequalities
$$
	2\pi
	\,\mathrm{inr}(K) 
	\, \mathrm{W}(K) 
	\le 
	\mathrm{S}(K) 
	\le 2\pi
	\,\mathrm{outr}(K)
	\,\mathrm{W}(K),
$$
where $\mathrm{inr}(K)$ and $\mathrm{outr}(K)$ denote the inradius and 
outradius of $K$.
The equality cases in these inequalities correspond to solutions of 
variational problems: among all convex bodies in $\mathbb{R}^3$ with 
unit outradius (inradius), which maximize (minimize) the ratio
of surface area to mean width?
Even though the two inequalities arise in a completely symmetric
manner, their extremals are very different.
\smallskip
\begin{enumerate}[$\bullet$]
\itemsep\medskipamount
\item A convex body in $\mathbb{R}^3$ with unit outradius maximizes
the ratio
of surface area to mean width if and only if it is a translate of the 
Euclidean unit ball $B$.
\item A convex body in $\mathbb{R}^3$ with unit inradius minimizes the 
ratio
of surface area to mean width if and only if it is a translate of a cap 
body of $B$ (Figure \ref{fig:cap}).
\end{enumerate}
\smallskip
This special case of Problem \ref{prb:mink} is due to Favard 
\cite{Fav33}, see \cite[Theorem 7.6.17]{Sch14}.
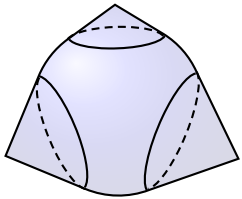
\begin{figure}
\centering
\begin{tikzpicture}
\begin{scope}[rotate=337.5,scale=1.15]

  \shade[ball color = blue, opacity = 0.15] (0,-1) arc 
(270:313.2:1) -- (1.5,0) -- (0.666,0.745) arc (46.8:80.3:1) 
-- (-.5,1.1) -- (-.853,.521) arc (148.6:180:1) -- (-1,-1) -- 
(0,-1);

  \draw[thick, densely dashed] (0,-1) [rotate=-225] arc 
[start angle = 170, end angle = 10,
    x radius = .707, y radius = .2];

  \draw[thick] (-1,0) [rotate=-45] arc [start angle = 170, 
end angle = 10,
    x radius = .707, y radius = .2];

  \draw[thick] (0.666,-0.745) [rotate=90] arc [start angle = 
170, end angle = 10,
    x radius = .745, y radius = .2];

  \draw[thick, densely dashed] (0.666,0.745) [rotate=-90] arc 
[start angle = 170, end angle = 10,
    x radius = .745, y radius = .2];

  \draw[thick] (.169,.986) [rotate=204.5] arc [start angle = 
170, end angle = 10,
    x radius = .561, y radius = .15];

  \draw[thick, densely dashed] (-.853,.521) [rotate=24.5] arc 
[start angle = 170, end angle = 10,
    x radius = .561, y radius = .15];

  \draw[thick] (0,-1) arc (270:313.2:1) -- (1.5,0) -- 
(0.666,0.745) arc (46.8:80.3:1) -- (-.5,1.1) -- (-.853,.521) 
arc (148.6:180:1) -- (-1,-1) -- (0,-1);

\end{scope}
\end{tikzpicture}
\caption{A cap body of $B$: i.e., the convex hull of 
$B$ with a finite or countable number of points so that the cones
emanating from the points are disjoint.\label{fig:cap}}
\end{figure}
\end{example}

The following reformulation of Problem \ref{prb:mink} will be the main
focus of this paper.
Recall that the support function $h_K$ of
a convex body $K$ in $\mathbb{R}^n$ is defined by
$$
	h_K(u) = \sup_{x\in K} \langle u,x\rangle.
$$
Geometrically, $h_K(u)$ is the (signed) distance to the origin of the
supporting hyperplane of $K$ with outer normal direction $u\in S^{n-1}$.
Mixed volumes can be expressed in terms of the support function of
one of the bodies as
$$
	\V_n(K,C_1,\ldots,C_{n-1}) =
	\frac{1}{n}\int h_K\,d\SM_{C_1,\ldots,C_{n-1}},
$$
where $\SM_{C_1,\ldots,C_{n-1}}$ is the \emph{mixed area measure}
on $S^{n-1}$. Now note that if $K\subseteq L$, then $h_L-h_K\ge 0$ 
pointwise, and thus the following are equivalent:
\smallskip
\begin{enumerate}[a.]
\itemsep\medskipamount
\item Equality holds in \eqref{eq:mono}.
\item $K$ and $L$ have the same supporting hyperplanes with outer normal 
direction in the support of the measure $\SM_{C_1,\ldots,C_{n-1}}$.
\end{enumerate}
\smallskip
We can therefore reformulate Problem \ref{prb:mink} as follows.

\begin{prb}
\label{prb:schneider}
Provide a geometric characterization of $\supp \SM_{C_1,\ldots,C_{n-1}}$.
\end{prb}

Beside its fundamental interest in convex geometry, Problem 
\ref{prb:schneider} is closely connected with the long-standing problem of 
providing a complete characterization of the equality cases of the 
Alexandrov--Fenchel inequality for general convex bodies 
\cite{Sch85,SvH23}, and has further implications that will be discussed in 
\S\ref{sec:impl}.

While Problem \ref{prb:schneider} has to date been resolved only for 
special classes of convex bodies, a precise conjectural picture that is 
consistent with all known cases was put forward by Schneider in 1985 
\cite{Sch85}. We presently discuss Schneider's conjecture and recall the 
previously known results in this direction.

\subsection{Schneider's conjecture}
\label{sec:schneider}

To motivate the conjectural picture, it is instructive to first consider 
the special case that $C_1,\ldots,C_{n-1}$ are all polytopes.
In this setting, the mixed area measure $\SM_{C_1,\ldots,C_{n-1}}$ is 
atomic with \cite[(5.22)]{Sch14}
\begin{equation}
\label{eq:mixedapoly}
	\SM_{C_1,\ldots,C_{n-1}}(\{u\}) =
	\V_{n-1}(F(C_1,u),\ldots,F(C_{n-1},u))
\end{equation}
for all $u\in S^{n-1}$, where $F(C,u)$ denotes the exposed face of $C$ 
with normal direction $u$ (see \S\ref{sec:facesetc}).
Thus to understand the support of the mixed area measure of polytopes,
it suffices to understand when mixed volumes are positive. The latter 
is classical and admits an elementary proof, see 
\cite[Theorem 5.1.8]{Sch14}.

\begin{fact}
\label{fact:mvpost}
The following are equivalent for convex bodies $C_1,\ldots,C_n$ in 
$\mathbb{R}^n$.
\begin{enumerate}[a.]
\itemsep\smallskipamount
\item $\V_n(C_1,\ldots,C_n)>0$.
\item There exist segments $I_i\subseteq C_i$, $i\in[n]$ with
linearly independent directions.
\item $\dim\big(\sum_{i\in I}C_i\big)\ge |I|$ for all $I\subseteq[n]$.
\end{enumerate}
\end{fact}

The above ingredients suffice to provide a satisfactory answer 
to Problem \ref{prb:schneider} for polytopes: combining
\eqref{eq:mixedapoly} and Fact \ref{fact:mvpost} yields
$$
	\supp \SM_{C_1,\ldots,C_{n-1}} =
	\big\{ u\in S^{n-1}:
	\dim\big(F(C_I,u)\big)\ge |I|
	\text{ for all }I\subseteq[n-1]\big\}
$$
whenever $C_1,\ldots,C_{n-1}$ are convex polytopes in $\mathbb{R}^n$,
where we define 
$$
	C_I = {\textstyle \sum_{i\in I} C_i}.
$$
However, this conclusion fails to extend to general convex 
bodies: for example, if $C_1,\ldots,C_{n-1}$ are strictly convex, then 
$\dim(F(C_I,u))=0$ for all $u,I$.

It was realized by Schneider \cite{Sch85} that the characterization 
remains meaningful for general convex bodies if $F(C,u)$ is 
replaced by the ``tangent space'' $T(C,u)^\perp$ of $C$ with normal 
direction $u$, where the touching cone $T(C,u)$ is defined as the unique 
face of a normal cone of $C$ so that $T(C,u)$ contains $u$ in its relative 
interior. Equivalent formulations of the following
definition are given in \S\ref{sec:extreme}.

\begin{defn}
\label{defn:extreme}
Let $C_1,\ldots,C_{n-1}$ be convex bodies in $\mathbb{R}^n$. Then 
$u\in S^{n-1}$ is called \emph{$(C_1,\ldots,C_{n-1})$-extreme}
if $\dim\big(T(C_I,u)^\perp\big)\ge |I|$
for all $I\subseteq[n-1]$.
\end{defn}

\begin{conj}[Schneider]
\label{conj:schneider}
Let $C_1,\ldots,C_{n-1}$ be convex bodies in $\mathbb{R}^n$. Then
$$
	\supp \SM_{C_1,\ldots,C_{n-1}} =
	\cl\big\{
	u\in S^{n-1}:u\text{ is }(C_1,\ldots,C_{n-1})\text{-extreme}
	\big\}.
$$
\end{conj}

\smallskip

To date, Conjecture \ref{conj:schneider} has been verified only for 
special classes of convex bodies. In particular, the conjecture is known
to hold in the following cases:
\begin{enumerate}[$\bullet$]
\itemsep\smallskipamount
\item $C_1,\ldots,C_{n-1}$ are convex polytopes, as explained above
\cite{Sch85};
\item $C_1=\cdots=C_k$ and 
$C_{k+1},\ldots,C_{n-1}$ are smooth and strictly convex 
\cite{Sch75,Sch85};\footnote{%
The strict convexity assumption is not needed and can be removed
along the lines of \cite[\S 14]{SvH23}.
}
\item $C_1,\ldots,C_{n-1}$ are zonoids \cite{Sch88,SvH23,HR24} or
``polyoids'' \cite{HR24};
\item various combinations of the above cases \cite{SvH23,HR24}.
\end{enumerate}
The proofs of all these results rely crucially on the special structure
of the bodies involved. While all known cases are consistent with
Conjecture \ref{conj:schneider}, none of the known results to date sheds 
any light on its validity for general convex bodies.

\subsection{Main results}

\subsubsection{Upper bound}

Our first main result proves one half of Conjecture~\ref{conj:schneider} 
in full generality: $\supp \SM_{C_1,\ldots,C_{n-1}}$ is always included in 
the closure of the set of $(C_1,\ldots,C_{n-1})$-extreme directions. We 
will in fact prove a significantly stronger result of which this is a 
consequence, as we now explain.

\begin{figure}
\centering
\begin{tikzpicture}

\begin{scope}[scale=3]
\fill[red!10!white] (-.866,0.5) -- (0,0) -- (-.866,-0.5);

\draw[thick,red] (0,0) to (-.866,0.5);
\draw[thick,blue,->] (0,0) to (-.433,0.25);
\draw[blue] (-.175,0.175) node {$\scriptstyle u$};

\draw[thick,color=black!30!white] (-0.25,-.433) to (0.25,.433);
\draw (0.3,.5) node {$\scriptstyle T(C,u)^\perp$};

\draw (-.5,0) node {$\scriptstyle N(C,F(C,u))$};

\draw[red] (-.866,0.55) node {$\scriptstyle T(C,u)$};

\draw[thick,fill=black!5!white] (-0.001,-0.002) to[out=60,in=90] (1.5,0) 
to[out=270,in=-60] (-0.001,0.002);

\draw (0.8,0) node {$\scriptstyle C$};

\end{scope} 
\end{tikzpicture}
\caption{Illustration of a normal direction $u$ of a convex body $C$ in 
$\mathbb{R}^2$ that is $C$-extreme but not $C$-exposed. In this case, the
``tangent space'' $T(C,u)^\perp$ is only tangent to the boundary of $C$
in one direction.\label{fig:exex}}
\end{figure}
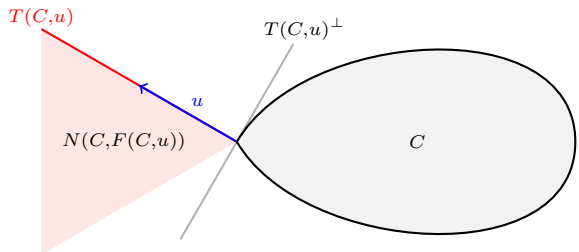

The notion of a ``tangent space'' of $C$ with normal direction $u$ is 
somewhat ambiguous in cases where $u$ does not lie in the relative 
interior of a normal cone of $C$. This situation is illustrated in Figure 
\ref{fig:exex}: in this case $T(C,u)^\perp$ is only tangent to the body in 
one direction. This situation can be avoided by replacing the touching 
cone $T(C,u)$ by the normal cone $N(C,F(C,u))$ of $C$ at $F(C,u)$ (that 
is, the smallest normal cone of $C$ that contains $u$) in Definition 
\ref{defn:extreme}.

\begin{defn}
\label{defn:exposed}
Let $C_1,\ldots,C_{n-1}$ be convex bodies in $\mathbb{R}^n$. Then 
$u\in S^{n-1}$ is called \emph{$(C_1,\ldots,C_{n-1})$-exposed}
if $\dim\big(N(C_I,F(C_I,u))^\perp\big)\ge |I|$
for all $I\subseteq[n-1]$.
\end{defn}

The following is the first main result of this paper.

\begin{thm}
\label{thm:mainupper}
For any convex bodies $C_1,\ldots,C_{n-1}$ in $\mathbb{R}^n$, we have
$$
	\SM_{C_1,\ldots,C_{n-1}}\big(
	\big\{ u\in S^{n-1}: u\text{ is not }(C_1,\ldots,C_{n-1})\text{-exposed}
	\big\}
	\big) = 0.
$$
\end{thm}

\smallskip

This result immediately implies the following.

\begin{cor}
\label{cor:mainupper}
For any convex bodies $C_1,\ldots,C_{n-1}$ in $\mathbb{R}^n$, we have
\begin{align*}
	\supp \SM_{C_1,\ldots,C_{n-1}} 
	&\subseteq
	\cl\big\{
	u\in S^{n-1}:u\text{ is }(C_1,\ldots,C_{n-1})\text{-exposed}
	\big\} \\
	&\subseteq
	\cl\big\{
	u\in S^{n-1}:u\text{ is }(C_1,\ldots,C_{n-1})\text{-extreme}
	\big\}.
\end{align*}
\end{cor}

\begin{proof}
By the definition of the support of a measure, the first
inclusion is equivalent to the statement that the \emph{interior} of the
set of non-$(C_1,\ldots,C_{n-1})$-exposed directions has
$\SM_{C_1,\ldots,C_{n-1}}$-measure zero. This follows trivially
from Theorem \ref{thm:mainupper}.
For the second inclusion, it suffices to note that as $T(C,u)$ is a
subset of $N(C,F(C,u))$, any $(C_1,\ldots,C_{n-1})$-exposed direction
is \emph{a fortiori} $(C_1,\ldots,C_{n-1})$-extreme.
\end{proof}

\subsubsection{Lower bound}

Our second main result fully resolves Conjecture \ref{conj:schneider} in 
$\mathbb{R}^3$. This will again follow as a consequence of a somewhat more 
general result.

\begin{thm}
\label{thm:mainlower}
For any convex bodies $K,L$ in $\mathbb{R}^n$, we have
$$
	\supp \SM_{K,\ldots,K,L} = \cl\big\{u\in S^{n-1}:
	u\text{ is }(K,\ldots,K,L)\text{-extreme}\big\}.
$$
In particular, this fully resolves Conjecture \ref{conj:schneider} in 
dimension $n=3$.
\end{thm}

The works of Minkowski \cite{Min03,Min11} were set exclusively in 
$\mathbb{R}^3$, and thus Theorem \ref{thm:mainlower} resolves Minkowski's 
monotonicity problem in its original setting. The obstacle to fully 
resolving Conjecture \ref{conj:schneider} in higher dimensions will be 
explained in \S\ref{sec:higher}.

\begin{rem}
Using the methods of \cite{Sch85,Sch88,SvH23,HR24}, one can readily 
generalize Theorem~\ref{thm:mainlower} by combining
it with previously known cases; e.g., one may consider
$\SM_{K,L,C_1,\ldots,C_{n-3}}$ where
$C_1,\ldots,C_{n-3}$ are smooth.
We do not spell out such variations on Theorem~\ref{thm:mainlower} as
they do not shed any new light on Conjecture \ref{conj:schneider}.
\end{rem}

\begin{rem}
\label{rem:smoothmystery}
While Theorem \ref{thm:mainlower} considers a special case of mixed 
area measures in $\mathbb{R}^n$, we emphasize the result holds for
\emph{arbitrary} convex bodies $K,L$, in contrast to previous results
that were restricted to special classes of bodies.

On the other hand, Theorem \ref{thm:mainlower}
is new even in the special case that the support functions of $K,L$ 
are smooth. At first sight, this setting would appear to be far easier 
than the general case, as the mixed area measure has a simple explicit 
formula: if $C_1,\ldots,C_{n-1}$ have 
smooth support functions, then
$$
	d\SM_{C_1,\ldots,C_{n-1}} =
	\D_{n-1}(D^2h_{C_1},\ldots,D^2h_{C_{n-1}})\,d\omega,
$$
where $\omega$ denotes the Lebesgue measure on $S^{n-1}$, 
$D^2h_C(u)$ denotes the restriction of the Hessian $\nabla^2h_C(u)$ in 
$\mathbb{R}^n$ to $u^\perp$, and the mixed 
discriminant $\D_n(M_1,\ldots,M_n)$ of $n$-dimensional matrices $M_i$ is 
defined analogously to mixed volumes as
$$
	\det(\lambda_1M_1+\cdots+\lambda_nM_n) =
	\sum_{i_1,\ldots,i_n} \D_n(M_{i_1},\ldots,M_{i_n})\,
	\lambda_{i_1}\cdots\lambda_{i_n};
$$
see, e.g., \cite[\S 2]{SvH19}. Thus using the analogue of
Fact \ref{fact:mvpost} for mixed discriminants due to Panov
\cite[Theorem 1]{Pan85} yields the characterization
$$
	\supp \SM_{C_1,\ldots,C_{n-1}} =
	\cl\big\{ u\in S^{n-1}:
	\mathrm{rank}\big(
	D^2h_{C_I}(u) \big)\ge |I|
	\text{ for all }I\subseteq[n-1]\big\}
$$
for any convex bodies $C_1,\ldots,C_{n-1}$ in $\mathbb{R}^n$
with smooth support functions. However, this is not a satisfactory 
characterization since it is analytic rather than geometric in nature.
The problem of deducing Conjecture \ref{conj:schneider} from this
analytic description is closely related to classical problems on the 
solution of homogeneous Monge-Amp\`ere equations, which will be discussed in
\S\ref{sec:nirenberg} below. At present, Theorem \ref{thm:mainlower} 
is the only general result in this direction even in the case of smooth 
support functions.
\end{rem}

\subsection{Implications}
\label{sec:impl}

\subsubsection{Extreme and exposed directions}

An unexpected implication of Theorem \ref{thm:mainupper} is that it sheds 
light on the relation between extreme and exposed directions.

\begin{cor}
\label{cor:exposed}
If Conjecture \ref{conj:schneider} holds for given 
convex bodies $C_1,\ldots,C_{n-1}$ in $\mathbb{R}^n$, then every 
$(C_1,\ldots,C_{n-1})$-extreme direction $u\in S^{n-1}$ is a limit
$u=\lim_k u_k$ of
$(C_1,\ldots,C_{n-1})$-exposed directions $u_k\in S^{n-1}$.
\end{cor}

\begin{proof}
Theorem \ref{thm:mainupper} and
Conjecture \ref{conj:schneider} imply that the closures of the sets
of $(C_1,\ldots,C_{n-1})$-extreme and $(C_1,\ldots,C_{n-1})$-exposed
directions coincide.
\end{proof}

Thus, for example, Theorem \ref{thm:mainlower} implies that
$$
	\cl\big\{
	u\in S^{n-1}:u\text{ is }(K,L)\text{-extreme}
	\big\} 
	=	
	\cl\big\{
	u\in S^{n-1}:u\text{ is }(K,L)\text{-exposed}
	\big\} 
$$
for any convex bodies $K,L$ in $\mathbb{R}^3$. 

Another special case of Corollary \ref{cor:exposed} gives a new proof of 
the following result.
A vector $u\in S^{n-1}$ is called an \emph{$r$-extreme 
normal direction} of $K$ if $\dim T(K,u)\le r+1$, and
is called an \emph{$r$-exposed normal direction} if $\dim N(K,F(K,u))\le 
r+1$. It is clear that these definitions are equivalent to 
$u$ being $(K[n-r-1],B[r])$-extreme and $(K[n-r-1],B[r])$-exposed,
respectively, where $C[k]$ denotes that $C$ is repeated $k$ times.
Since Conjecture \ref{conj:schneider} holds in this setting
\cite{Sch75}, it follows that any $r$-extreme normal direction of a convex 
body is the limit of $r$-exposed normal directions. This fact was 
previously established by a direct argument in \cite[Theorem 
2.2.9]{Sch14}.

\subsubsection{Mixed Hessian measures}

We now describe an analogue of mixed area measures for convex functions 
rather than convex bodies. Such measures are readily obtained \cite{HMU24} 
by polarizing the Monge-Amp\`ere measure \cite[Chapter 1]{Gut16}, and have 
been investigated (in a somewhat more general setting) by Trudinger and 
Wang \cite{TW02} as part of their study of nonlinear Dirichlet problems.

In the following, we fix an open convex set 
$\Omega\subseteq\mathbb{R}^n$, and denote by $\mathrm{Conv}(\Omega)$
the set of all proper convex functions $f$ on $\mathbb{R}^n$ with
$\Omega\subseteq\mathop{\mathrm{dom}} f$. If 
$f_1,\ldots,f_n\in\mathrm{Conv}(\Omega)$ are smooth, their 
\emph{mixed Hessian measure} is the measure on $\Omega$
with density
$$
	\frac{d\HH_{f_1,\ldots,f_n}}{dx} =
	\D_n(\nabla^2f_1,\ldots,\nabla^2f_n)	
$$
with respect to the Lebesgue measure. The definition of 
$\HH_{f_1,\ldots,f_n}$ can be extended to arbitrary 
$f_1,\ldots,f_n\in\mathrm{Conv}(\Omega)$ by continuity 
\cite[Theorem 2.4]{TW02}.

For any $f\in\mathrm{Conv}(\Omega)$ and $x\in\Omega$, denote by $L(f,x)$ 
the largest convex subset of $\Omega$ on which $f$ is affine and that has 
$x$ in its relative interior. The set $L(f,x)$ is the analogue for convex 
functions of the notion of a touching cone for convex bodies (indeed, the 
support function $h_C$ is affine on every normal cone of $C$; for the 
precise relation between mixed Hessian measures and mixed area measures, 
see \cite[Corollary 4.2]{HMU24}). In the following, we let $\bar 
L(f,x)=\mathrm{span}\{L(f,x)-x\}$.

We can now transcribe Conjecture \ref{conj:schneider} to the 
setting of mixed Hessian measures. 

\begin{defn}
\label{defn:fcnextreme}
Let $f_1,\ldots,f_n\in\mathrm{Conv}(\Omega)$. Then a point $x\in\Omega$
is called \emph{$(f_1,\ldots,f_n)$-extreme} if
$\dim\big(\bar L(f_I,x)^\perp \big)\ge |I|$
for all $I\subseteq[n]$.
\end{defn}

\begin{conj}
\label{conj:fcnschneider}
Let $f_1,\ldots,f_n\in\mathrm{Conv}(\Omega)$. Then
$$
	\supp \HH_{f_1,\ldots,f_{n}} =
	\cl\big\{
	x\in\Omega:x\text{ is }(f_1,\ldots,f_{n})\text{-extreme}
	\big\}.
$$
\end{conj}

\smallskip

It can be shown using results of Hug--Mussnig--Ulivelli \cite{HMU24} that
Conjecture \ref{conj:schneider} and Conjecture \ref{conj:fcnschneider} are
in fact equivalent; see \S\ref{sec:mixedhess}. By the same technique, we 
can readily transcribe the main results of this paper to mixed Hessian 
measures.

\begin{cor}
\label{cor:mixedhess}
For any $f_1,\ldots,f_n\in\mathrm{Conv}(\Omega)$, we have
$$
	\supp \HH_{f_1,\ldots,f_{n}} \subseteq
	\cl\big\{
	x\in\Omega:x\text{ is }(f_1,\ldots,f_{n})\text{-extreme}
	\big\}.
$$
Moreover, for any $f,g\in\mathrm{Conv}(\Omega)$, we have
$$
	\supp \HH_{f,\ldots,f,g} =
	\cl\big\{
	x\in\Omega:x\text{ is }(f,\ldots,f,g)\text{-extreme}
	\big\}.	
$$
In particular, this fully resolves Conjecture \ref{conj:fcnschneider}
in dimension $n=2$.
\end{cor}

\subsubsection{A mixed Hartman--Nirenberg--Pogorelov theorem}
\label{sec:nirenberg}

A classical fact that dates back to Monge \cite[\S 23.7]{Kli90}, and in 
more precise form to Hartman and Nirenberg \cite{HN59} and Pogorelov 
\cite[\S IX.4]{Pog73}, is that a surface with vanishing Gauss curvature is 
a ruled surface: it is foliated by straight lines and planar regions.

More concretely, let $f:D\to\mathbb{R}$ be a smooth (not necessarily 
convex) function on $D\subset\mathbb{R}^2$. The graph of $f$ defines a 
surface in $\mathbb{R}^3$, whose Gauss curvature vanishes if
\begin{equation}
\label{eq:homomonge}
	\det(\nabla^2 f)=0\quad\text{on}\quad D,
\end{equation}
that is, if $f$ it a solution of the homogeneous Monge-Amp\`ere equation
on $D$. That $\det(\nabla^2f(x))=0$ means that $\nabla^2f(x)$ has a 
nontrivial kernel---that is, the second derivative of $f$ vanishes in 
some direction---but this local condition does not in itself imply that 
$f$ must be affine in that direction. It is a nontrivial fact that when 
$\det(\nabla^2 f)$ vanishes in an open domain $D$, this local condition 
can be integrated to yield the global property that $D$ is foliated by 
regions on which $f$ is affine. The following formulation is due 
to Hartman and Nirenberg \cite[\S 3]{HN59}.

\begin{thm}[Hartman--Nirenberg]
\label{thm:hn}
Let $D\subset\mathbb{R}^2$ be an open connected set,
let $f:D\to\mathbb{R}$ be of class $C^2$, and suppose that 
\eqref{eq:homomonge} holds. Let
$$
	R = \{x\in D: \nabla^2f=0\text{ in a neighborhood of }x\}.
$$
Then for every $x\in D\backslash R$, there is an 
affine line $x\in L\subset\mathbb{R}^2$ so that $f$ is affine on
the connected component $I$ of $L\cap D$ that contains $x$, and 
$I\cap R=\varnothing$.
\end{thm}

It is clear that $f$ is affine on every connected component of the open 
set $R$; these are the planar regions of the surface defined by $f$. 
Theorem \ref{thm:hn} states that the non-planar part of the surface is 
foliated by lines that extend to the boundary. Such a result for 
non-smooth surfaces is due to Pogorelov \cite[\S IX.4]{Pog73}.

The problems investigated in this paper may be viewed as mixed analogues 
of these classical results. More concretely, given two smooth functions 
$f:D\to\mathbb{R}$ and $g:D\to\mathbb{R}$ on a domain 
$D\subseteq\mathbb{R}^2$, we aim to characterize when
\begin{equation}
\label{eq:mixedmonge}
	\D_2(\nabla^2f,\nabla^2g) = 0 \quad\text{on}\quad D.
\end{equation}
When $f=g$, this is precisely \eqref{eq:homomonge}. However, 
the case $f\ne g$ can be of a very different nature, as
$\D_2(M_1,M_2)=0$ need not have any implication
for the kernels of $M_1,M_2$: for example, \eqref{eq:mixedmonge} holds
when $f(x)=\|x\|^2$ and $g$ is any harmonic function on $D$, neither of 
which define ruled surfaces.

When $M_1,M_2$ are positive semidefinite, however, their kernels do 
characterize when $\D_2(M_1,M_2)=0$: this holds if and only if 
$M_1=0$, or $M_2=0$, or $\ker M_1=\ker M_2$ with $\dim\ker M_i=1$ 
\cite{Pan85}. Therefore, by analogy with the local-to-global phenomenon 
captured by Theorem \ref{thm:hn}, it is natural to conjecture that 
\eqref{eq:mixedmonge} has the following solution for \emph{convex} $f,g$: 
each $x\in D$ is either contained in a planar region of $f$ or of $g$, or 
in an affine line on which $f$ and $g$ are simultaneously affine; see 
Figure \ref{fig:hn}. That this is the case is a consequence of 
Theorem \ref{thm:mainlower}.

\begin{cor}
\label{cor:mixedhn}
Let $\Omega\subseteq\mathbb{R}^2$ be an open convex set and
$D\subseteq\Omega$ be an open connected set.
Let $f,g\in\mathrm{Conv}(\Omega)$ be of class $C^2$, and suppose that 
\eqref{eq:mixedmonge} holds. Let
\begin{align*}
	R = 
	&\{x\in D: \nabla^2f=0\text{ in a neighborhood of }x\} 
	\cup \mbox{}
\\	&\{x\in D: \nabla^2g=0\text{ in a neighborhood of }x\}.
\end{align*}
Then for every $x\in D\backslash R$, there is an affine line $x\in 
L\subset\mathbb{R}^2$ so that $f$ and $g$ are both affine on the connected
component $I$ of $L\cap D$ that contains $x$, and $I\cap R=\varnothing$.
\end{cor}
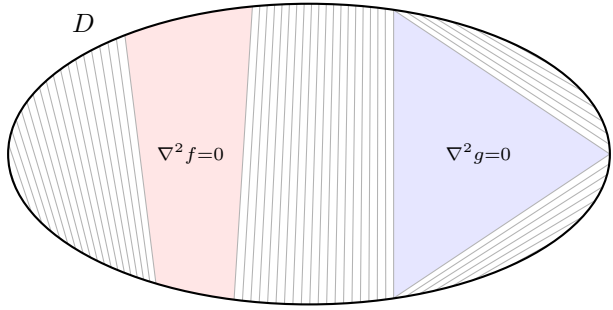
\begin{figure}
\centering
\begin{tikzpicture}

\begin{scope}
\clip (0,0) ellipse (4cm and 2cm);

\begin{scope}
\clip (1.125,-2) -- (1.125,2) -- (4,0);
\fill[color=blue!10!white] (1,-2) -- (4,0) --  (1,2);
\end{scope}

\draw[color=black!30!white] (1.125,-2) to (1.125,2);
\draw[color=black!30!white] (1,-2) to (4,0);
\draw[color=black!30!white] (1,2) to (4,0);

\fill[color=red!10!white] (-2,-2) -- (-2.5,2)--(-0.75,2)--(-1,-2);

\draw[color=black!30!white] (-2,-2) to (-2.5,2);
\draw[color=black!30!white] (-1,-2) to (-0.75,2);

\foreach \r in {1,...,30}{
    \draw[color=black!30!white] (-2-0.075*\r,-2) to (-2.5-0.125*\r,2);
}

\foreach \r in {1,...,18}{
    \draw[color=black!30!white] (-1+0.11225*\r,-2) to (-0.75+0.098*\r,2);
}

\foreach \r in {1,...,10}{
    \draw[color=black!30!white] (1+0.125*\r,-2) to (4,-0.125*\r);
}

\foreach \r in {1,...,10}{
    \draw[color=black!30!white] (1+0.125*\r,2) to (4,0.125*\r);
}

\end{scope}

\draw (-1.575,0) node {$\scriptstyle\nabla^2f=0$};
\draw (2.25,0) node {$\scriptstyle\nabla^2g=0$};

\draw[thick] (0,0) ellipse (4cm and 2cm);

\draw (-3,1.75) node {$D$};

\end{tikzpicture}
\caption{Illustration of of Corollary \ref{cor:mixedhn}. Outside the
planar regions of $f$ and $g$, the domain $D$ is foliated by lines on 
which $f$ and $g$ are simultaneously affine.
\label{fig:hn}}
\end{figure}

We have formulated Corollary \ref{cor:mixedhn} for smooth convex functions 
to emphasize the analogy with Theorem \ref{thm:hn}. However, the result is 
an immediate consequence of the following result in the nonsmooth case.

\begin{cor}
\label{cor:mixedhnnonsm}
Let $\Omega\subseteq\mathbb{R}^2$ be an open convex set and
$D\subseteq\Omega$ be an open connected set.
Let $f,g\in\mathrm{Conv}(\Omega)$ satisfy $\HH_{f,g}(D)=0$, and define
$$
	R = \{x\in D: \dim L(f,x)=2\text{ or }\dim L(g,x)=2\}.
$$
Then for every $x\in D\backslash R$, there is an affine line $x\in
L\subset\mathbb{R}^2$ so that $f$ and $g$ are both affine on the connected
component $I$ of $L\cap D$ that contains $x$, and $I\cap R=\varnothing$.
\end{cor}

\subsection{On higher dimensions}
\label{sec:higher}

The proof of Theorem \ref{thm:mainupper} is based on an argument with a 
geometric measure theory flavor that will be explained in 
\S\ref{sec:pfupper}. This argument works in the same manner in any 
dimension. In contrast, Theorem~\ref{thm:mainlower} is restricted to 
special mixed area measures that fully capture only the three-dimensional 
case. The reason for this is not merely technical in nature: there is a 
fundamental obstacle that arises in higher dimensions, as we presently 
explain.

The proof of Theorem~\ref{thm:mainlower} is based on the following simple 
observation that will be proved in \S\ref{sec:extreme} below. Here 
$\proj_E$ denotes the orthogonal projection onto $E$.

\begin{lem}
\label{lem:extremeproj}
Let $K,C_1,\ldots,C_{n-2}$ be convex bodies in $\mathbb{R}^n$, and let 
$u\in S^{n-1}$ be a $(K,C_1,\ldots,C_{n-2})$-extreme direction. Then there 
exists $v\in T(K,u)^\perp$ so that $u$ is a 
$(\proj_{v^\perp}C_1,\ldots,\proj_{v^\perp}C_{n-2})$-extreme direction.
\end{lem}

Lemma \ref{lem:extremeproj} states that the property of being an extreme 
direction is preserved under projection. Therefore, if it were true that 
\begin{equation}
\label{eq:ifonly}
	u\in \supp \SM_{\proj_{v^\perp}C_1,\ldots,\proj_{v^\perp}C_{n-2}}
	\quad\stackrel{?}{\Longrightarrow}\quad
	u\in \supp \SM_{K,C_1,\ldots,C_{n-2}},
\end{equation}
then the lower bound in Conjecture \ref{conj:schneider} would follow by 
induction on the 
dimension (and thus the full conjecture would follow by Theorem 
\ref{thm:mainupper}). Unfortunately, it turns out that \eqref{eq:ifonly} 
does not always hold. In fact, we will provide an
essentially complete understanding of the behavior of the support under 
projection.

\begin{thm}
\label{thm:mainproj}
Let $K,C_1,\ldots,C_{n-2}$ be convex bodies in 
$\mathbb{R}^n$.
\smallskip
\begin{enumerate}[a.]
\itemsep\medskipamount
\item If $\dim T(K,u)=1$ or $\dim T(K,u)=n-1$, the following
holds: if there exists $v\in T(K,u)^\perp$ so that
$u\in \supp \SM_{\proj_{v^\perp}C_1,\ldots,\proj_{v^\perp}C_{n-2}}$, then
$u\in \supp \SM_{K,C_1,\ldots,C_{n-2}}$.
\item If $1<\dim T(K,u)<n-1$, it can be the case
that $u\in \supp \SM_{\proj_{v^\perp}C_1,\ldots,\proj_{v^\perp}C_{n-2}}$ 
for every $v\in T(K,u)^\perp$, but
$u\not\in \supp \SM_{K,C_1,\ldots,C_{n-2}}$.
\smallskip
\end{enumerate}
\end{thm}

Only the case $\dim T(K,u)=1$ will be needed 
for the proof of Theorem~\ref{thm:mainlower}; this case will be proved in 
\S\ref{sec:pflower}. The case $\dim T(K,u)=n-1$ and part \textit{b.}\ are 
included to clarify the basic obstruction to extending the lower bound to 
higher dimensions; their proofs are postponed until 
\S\ref{sec:obstruction}.

Theorem \ref{thm:mainproj} shows that performing induction on the 
dimension by projection must fail in dimensions $n\ge 4$.
Thus further progress on the lower bound in Conjecture 
\ref{conj:schneider} will likely require a fundamentally new ingredient.

\subsection{Organization of this paper}

The remainder of this paper is organized as follows. In 
\S\ref{sec:prelim}, we recall some basic notions of convex geometry that 
appear throughout the paper, and we establish some elementary facts that 
will be used in the sequel. The main results of this paper, Theorems 
\ref{thm:mainupper} and \ref{thm:mainlower}, are proved in 
\S\ref{sec:pfupper} and \S\ref{sec:pflower}, respectively. The 
corresponding results for mixed Hessian measures and the mixed 
Hartman-Nirenberg-Pogorelov problem are developed in 
\S\ref{sec:mixedhess}. Finally, the behavior of the support of mixed area 
measures in higher dimensions is discussed in
\S\ref{sec:obstruction}.

\section{Preliminaries}
\label{sec:prelim}

The aim of this section is to recall a number of basic notions of convex 
geometry and to establish some elementary facts that will be used 
throughout this paper. Our standard reference on convexity is the 
excellent monograph \cite{Sch14}.

\subsection{Basic notions}

By a \emph{convex body} we mean a nonempty compact convex set.
We denote by $\mathcal{K}^n$ the set of all convex bodies in
$\mathbb{R}^n$, by $\mathcal{K}_n^n$ the set of all convex bodies in
$\mathbb{R}^n$ with nonempty interior, and $\mathcal{K}^n_{(o)}$ the set 
of all convex bodies in $\mathbb{R}^n$ that contain the origin in their
interior.

For $K\in\mathcal{K}^n$, we denote by $\relint K$ the 
relative interior and by $\relbd K$ the relative boundary of $K$, that is, 
the interior (boundary) of $K$ viewed as a convex body in the affine hull 
of $K$. The polar dual body of $K\in\mathcal{K}^n_{(o)}$ is defined by
$$
	K^\circ = \{u\in\mathbb{R}^n:\langle u,x\rangle\le 1
	\text{ for all }x\in K\}.
$$
The polar is an involution $K^{\circ\circ}=K$. Moreover,
$$
	h_{K^\circ}(x) = \|x\|_K\quad\text{for all }x\in\mathbb{R}^n,
$$
where the Minkowski functional of $K$ is defined by
$\|x\|_K = \inf\{\lambda>0:x\in \lambda K\}$.

For any $u\in\mathbb{R}^n$ and $t\in\mathbb{R}$, we define
\begin{align*}
	H^+_{u,t} &= \{x\in\mathbb{R}^n : \langle u,x\rangle \ge t\},\\
	H^-_{u,t} &= \{x\in\mathbb{R}^n : \langle u,x\rangle \le t\},\\
	H_{u,t} &= \{x\in\mathbb{R}^n : \langle u,x\rangle = t\}.
\end{align*}
For $u\in S^{n-1}$ and $\varepsilon>0$, 
$B(u,\varepsilon)$ denotes the open $\varepsilon$-ball in $S^{n-1}$ with 
center $u$.

\subsection{Facial structure}
\label{sec:facesetc}

\subsubsection{Faces and exposed faces}

A \emph{face} $F$ of a convex body $K\in\mathcal{K}^n$ is a convex subset 
of $K$ such that $x,y\in K$ and $\frac{x+y}{2}\in F$ implies $x,y\in F$.
For $u\in S^{n-1}$,
$$
	F(K,u) = \{x\in K: \langle u,x\rangle=h_K(u)\} =
	K\cap H_{u,h_K(u)}
$$
is called the \emph{exposed face} of $K$ with normal direction $u$.
Exposed faces are additive under Minkowski addition, that is, 
\begin{equation}
\label{eq:faceaddition}
	F(K+L,u) = F(K,u) + F(L,u)
\end{equation}
for all $K,L\in\mathcal{K}^n$ and $u\in S^{n-1}$; see
\cite[Theorem 1.7.5]{Sch14}.

\subsubsection{Normal cones}
\label{sec:normalcones}

For any $x\in K$, the \emph{normal cone} of $K$ at $x$ is defined 
by
$$
	N(K,x) = \{u\in\mathbb{R}^n:\langle u,x\rangle=h_K(u)\}.
$$
This is precisely the dual notion to an exposed face.
Dual to \eqref{eq:faceaddition}, we 
have
\begin{equation}
\label{eq:normaladdition}
	N(K+L,x+y) = N(K,x)\cap N(L,y)
\end{equation}
for all $K,L\in\mathcal{K}^n$ and $x\in K$, $y\in L$; see
\cite[Theorem 2.2.1]{Sch14}.

The \emph{metric projection map} $p_K:\mathbb{R}^n\to K$ sends
each $x\in\mathbb{R}^n$ to the point in $K$ that is closest to $x$. 
The metric projection is closely related to the normal cone:
$$
	p_K^{-1}(x) = x + N(K,x)
$$
for every $x\in K$; see \cite[p.\ 81]{Sch14}.

Let $F$ be a face of $K$. Then every point $x\in\relint F$ yields the same 
normal cone $N(K,x)$, which will also be denoted as $N(K,F)$; see 
\cite[p.\ 83]{Sch14}. Since relative interior distributes over Minkowski 
addition \cite[Corollary 6.6.2]{Roc70}, the 
property \eqref{eq:normaladdition} extends
to $N(K+L,F(K+L,u))=N(K,F(K,u))\cap N(L,F(L,u))$.

\subsubsection{Touching cones}
\label{sec:touchingcones}

For any $u\in S^{n-1}$, the \emph{touching cone} $T(K,u)$ is defined as 
the unique face of $N(K,F(K,u))$ that contains $u$ in its relative 
interior. The distinction between touching and normal cones
is illustrated in Figure \ref{fig:exex}.

\begin{rem}
\label{rem:gentouch}
Since $N(K,F(K,u))$ is itself a face of every normal cone 
$N(K,x)$ that contains $u$, the touching cone $T(K,u)$ can be equivalently 
defined as the unique face of any normal cone of $K$ so that $T(K,u)$ 
contains $u$ in its relative interior.
\end{rem}

Just as normal cones are dual to exposed faces, touching cones are dual 
to faces. This duality can be made precise as follows \cite[\S 
1.2.3]{Wei12}. For any $K\in\mathcal{K}^n_{(o)}$, the map
$F\mapsto \mathbb{R}_+F = \{\lambda x:\lambda\ge 0,x\in F\}$
defines a bijection
$$
	\{\text{faces of }K^\circ\}
	\stackrel{\sim}{\longrightarrow} 
	\{\text{touching cones of }K\}.
$$
The restriction 
of this map to the set of exposed faces of $K^\circ$ defines a bijection
$$
	\{\text{exposed faces of }K^\circ\}
	\stackrel{\sim}{\longrightarrow} 
	\{\text{normal cones of }K\}.
$$
The inverse map is given by $T\mapsto T\cap \bd K^\circ$ for any touching 
cone $T$ of $K$.

\subsection{Properties of touching cones}

A number of elementary properties of touching cones will be needed 
throughout this paper. We record their proofs here. We begin with the 
following simple observation.

\begin{lem}
\label{lem:touchingface}
Let $v\in T(K,u)$. Then $T(K,v)$ is the unique face of $T(K,u)$ that 
contains $v$ in its relative interior.
\end{lem}

\begin{proof}
By definition, $T(K,u)$ is a face of $N(K,x)$ for some $x\in K$.
Let $F$ be the unique face of $T(K,u)$ that contains $v$ in its relative
interior. Then $F$ is also a face of $N(K,x)$ \cite[Theorem 
2.1.1]{Sch14}, and thus $F=T(K,v)$ by Remark \ref{rem:gentouch}.
\end{proof}

Our next observation is the analogue of \eqref{eq:normaladdition}
for touching cones.

\begin{lem}
\label{lem:touchingsum}
Let $K,L\in \mathcal{K}^n$ and $u\in S^{n-1}$. Then
$$
	T(K+L,u) = T(K,u) \cap T(L,u).
$$
\end{lem}

\begin{proof}
We first note that $u$ is in the relative interior of $T(K,u) \cap 
T(L,u)$, as the relative interior distributes over finite intersections
\cite[Theorem 6.5]{Roc70}.  On the other hand, let $T(K,u)$ be a face of
$N(K,x)$, and let $T(L,u)$ be a face of $N(L,y)$. Then clearly
$a,b\in N(K,x)\cap N(L,y)$ with $\frac{a+b}{2}\in
T(K,u) \cap T(L,u)$ implies that $a,b\in T(K,u) \cap T(L,u)$. Thus we have 
shown 
that $T(K,u) \cap T(L,u)$ is a face of $N(K+L,x+y)=N(K,x)\cap N(L,y)$,
which concludes the proof by Remark \ref{rem:gentouch}.
\end{proof}

Next, we clarify the behavior of touching cones under projection.

\begin{lem}
\label{lem:touchingproj}
Let $K\in \mathcal{K}^n$ and $u,v\in S^{n-1}$ with $u\in v^\perp$. Then
$$
	T(\proj_{v^\perp}K,u) = T(K,u) \cap v^\perp,
$$
where we view $P_{v^\perp}K$ as a convex body in $v^\perp$.
\end{lem}

\begin{proof}
Let $T(K,u)$ be a face of $N(K,x)$, and note that
$$
	N(\proj_{v^\perp}K,\proj_{v^\perp}x) =
	\{w\in v^\perp: \langle w,x\rangle = h_K(w)\} =
	N(K,x)\cap v^\perp.
$$
It follows by the same argument as in 
the proof of Lemma \ref{lem:touchingsum} that
$u$ is in the relative interior of
$T(K,u) \cap v^\perp$ and that $T(K,u) \cap v^\perp$ is a face of 
$N(K,x)\cap v^\perp$. This concludes the proof by Remark \ref{rem:gentouch}.
\end{proof}

We finally record an implication between normal and touching cone
inclusion.

\begin{lem}
\label{lem:touchincl}
Let $K,L\in\mathcal{K}^n$ and $x\in \bd K$, $y\in\bd L$
with $N(K,x)\subseteq N(L,y)$. Then
$T(K,u)\subseteq T(L,u)$ for every $u\in N(K,x)$.
\end{lem}

\begin{proof}
It follows by the argument in the proof 
of Lemma \ref{lem:touchingsum} 
and by Remark \ref{rem:gentouch}
that
$u$ is in the relative interior of $T(K,u)\cap T(L,u)$, and that
$T(K,u)\cap T(L,u)$ is a face of $N(K,x)\cap N(L,y)=N(K,x)$. Thus
$T(K,u)\cap T(L,u)=T(K,u)$.
\end{proof}

\subsection{Extreme directions}
\label{sec:extreme}

The notion of a $(C_1,\ldots,C_{n-1})$-extreme direction was defined in 
Definition \ref{defn:extreme}. This notion has a number of equivalent
formulations that are analogous to the equivalent conditions of
Fact \ref{fact:mvpost}.

\begin{lem}
\label{lem:extreme}
Let $C_1,\ldots,C_{n-1}\in\mathcal{K}^n$ and $u\in S^{n-1}$.
The following are equivalent:
\begin{enumerate}[a.]
\itemsep\smallskipamount
\item $u$ is $(C_1,\ldots,C_{n-1})$-extreme.
\item $\dim\big(\sum_{i\in I}T(C_i,u)^\perp \big)\ge |I|$
for all $I\subseteq[n-1]$.
\item There exist lines $L_i\subseteq T(C_i,u)^\perp$, 
$i\in[n-1]$ with linearly independent directions.
\end{enumerate} 
\end{lem}

\begin{proof}
By Lemma \ref{lem:touchingsum}, \textit{a.}\ states that
$$
	\dim\left({\textstyle\bigcap_{i\in I} T(C_i,u)}\right) =
	\dim\big(T(C_I,u)\big) \le n-|I|
$$
for all $I$, while \textit{b.}\ states that
$$
	\dim\left({\textstyle\bigcap_{i\in I}\mathop{\mathrm{span}} T(C_i,u)}\right) 
	\le n-|I|
$$
for all $I$. The equivalence of these two conditions follows from
the fact that $u$ is in the relative interior of each $T(C_i,u)$.
The equivalence of \textit{b.}\ and \textit{c.}\ is a general fact of
linear algebra that is proved, e.g., in \cite[Lemma 2.3]{Sch88}. 
\end{proof}

We emphasize that the counterpart of Lemma \ref{lem:extreme} 
for $(C_1,\ldots,C_{n-1})$-exposed directions does not hold: 
it may happen that the normal cones $N(C_i,F(C_i,u))$ do not have any 
common point in their relative interiors, and thus their intersection may 
have strictly smaller dimension than the intersection of their linear 
spans. That Definition \ref{defn:exposed} provides the natural 
notion of a $(C_1,\ldots,C_{n-1})$-exposed direction
will become evident from the proof of Theorem 
\ref{thm:mainupper}.

On the other hand, Lemma \ref{lem:extreme} extends readily to its 
counterpart for convex functions that was defined in Definition 
\ref{defn:fcnextreme}.

\begin{lem}
\label{lem:fcnextreme}
Let $\Omega\subseteq\mathbb{R}^n$ be an open convex set, and
let $f_1,\ldots,f_n\in\mathrm{Conv}(\Omega)$ and $x\in\Omega$.
The following are equivalent:
\begin{enumerate}[a.]
\itemsep\smallskipamount
\item $x$ is $(f_1,\ldots,f_n)$-extreme.
\item $\dim\big(\sum_{i\in I} \bar L(f_i,x)^\perp \big)\ge |I|$
for all $I\subseteq[n]$.
\item
There exist lines $L_i\subseteq \bar L(f_i,x)^\perp$,
$i\in[n]$ with linearly independent directions.
\end{enumerate} 
\end{lem}

\begin{proof}
The proof is identical to that of Lemma \ref{lem:extreme} using that
$$
	L(f_I,x) = 
	{\textstyle\bigcap_{i\in I} L(f_i,x)}.
$$
To verify this is the case, note that as
$$
	f_I\big(\tfrac{y+z}{2}) -
	\tfrac{f_I(y) + f_I(z)}{2} =
	\sum_{i\in I}
	\big(
	f_i\big(\tfrac{y+z}{2}) -
	\tfrac{f_i(y) + f_i(z)}{2} 
	\big)
$$
and each term in the sum is nonnegative, $f_I$ is affine on a set 
$A$ if and only if $f_i$ is affine on $A$ for every $i\in I$.
The claim therefore follows from the fact that the relative interior 
distributes over finite intersections \cite[Theorem 6.5]{Roc70}.
\end{proof}

We can now prove Lemma \ref{lem:extremeproj} in the introduction.

\begin{proof}[Proof of Lemma \ref{lem:extremeproj}]
Let $u$ be $(K,C_1,\ldots,C_{n-2})$-extreme.
By Lemma~\ref{lem:extreme}, there exist lines 
$L_0\subseteq T(K,u)^\perp$ and
$L_i\subseteq T(C_i,u)^\perp$, $i=1,\ldots,n-2$
with linearly independent directions. Let $L_0=\mathbb{R}v$.
Then the lines $\proj_{v^\perp}L_1,\ldots,\proj_{v^\perp}L_{n-2}$ have
linearly independent directions. Moreover, since each element of
$\proj_{v^\perp}L_i$ is a linear combination of an element of $L_i$ and 
$v$, we have
$$
	\proj_{v^\perp}L_i \subseteq
	T(C_i,u)^\perp + \mathbb{R}v =
	\big(T(C_i,u)\cap v^\perp\big)^\perp =
	T(\proj_{v^\perp}C_i,u)^\perp
$$
by Lemma \ref{lem:touchingproj}. Thus $u$ is 
$(\proj_{v^\perp}C_1,\ldots,\proj_{v^\perp}C_{n-2})$-extreme by
Lemma \ref{lem:extreme}.
\end{proof}

\subsection{Mixed volumes and mixed area measures}
\label{sec:mvma}

The most basic properties of mixed volumes $\V_n(C_1,\ldots,C_n)$ and 
mixed area measures $\SM_{C_1,\ldots,C_{n-1}}$ are that they are additive 
and $1$-homogeneous in each argument $C_i$, symmetric in their arguments 
$C_i$, and that they are nonnegative and translation invariant. For the 
theory of mixed volumes and mixed area measures, we refer to the 
monograph \cite{Sch14}.

Let $f=h_K-h_L$ be a difference of support functions of convex bodies 
$K,L\in\mathcal{K}^n$. By multilinearity,
we can uniquely extend the definitions of mixed volumes and mixed area 
measures to such functions by defining \cite[\S 5.2]{Sch14}
\begin{align*}
	\V_n(f,C_1,\ldots,C_{n-1}) &=
	\V_n(K,C_1,\ldots,C_{n-1})-\V_n(L,C_1,\ldots,C_{n-1}),\\
	\SM_{f,C_1,\ldots,C_{n-2}} &= 
	\SM_{K,C_1,\ldots,C_{n-2}} - \SM_{L,C_1,\ldots,C_{n-2}}.
\end{align*}
If $f_1,\ldots,f_n$ are differences of support functions, we can iterate 
this construction to define $\V_n(f_1,\ldots,f_n)$ and 
$\SM_{f_1,\ldots,f_{n-1}}$. Note that these functional extensions of
mixed volumes and mixed area measures
are no longer necessarily nonnegative. Any $f\in C^2(S^{n-1})$ is 
a difference of support functions \cite[Lemma 1.7.8]{Sch14}.

We now recall the behavior of mixed volumes under projection onto
a hyperplane \cite[Theorem 5.3.1]{Sch14}:
for any $v\in S^{n-1}$, we have
$$
	n\,\V_n([0,v],C_1,\ldots,C_{n-1}) =
	\V_{n-1}(\proj_{v^\perp}C_1,\ldots,\proj_{v^\perp}C_{n-1}).
$$
Since $h_{\proj_EC}=h_C|_E$, this implies that
$$
	n\,\V_n([0,v],f_1,\ldots,f_{n-1}) =
	\V_{n-1}(f_1|_{v^\perp},\ldots,f_{n-1}|_{v^\perp})
$$
for any differences of support functions $f_1,\ldots,f_{n-1}$.

In the opposite direction, we will require the following.

\begin{lem}
\label{lem:mixedareaproj}
Let $C_1,\ldots,C_{n-2}\in\mathcal{K}^n$, and let $u,v\in S^{n-1}$ with
$u\in v^\perp$. Then $u\in\supp 
\SM_{\proj_{v^\perp}C_1,\ldots,\proj_{v^\perp}C_{n-2}}$ implies that
$u\in \SM_{B,C_1,\ldots,C_{n-2}}$.
\end{lem}

\begin{proof}
The projection formula for mixed volumes implies
that \cite[Remark 8.6]{SvH23}
\begin{equation}
\label{eq:projmixedarea}
	(n-1)\,\SM_{[0,v],C_1,\ldots,C_{n-2}} =
	\SM_{\proj_{v^\perp}C_1,\ldots,\proj_{v^\perp}C_{n-2}}
\end{equation}
if we view the right-hand side as a measure on $\mathbb{R}^n$ that is 
supported in $v^\perp$. The result
follows as $\supp \SM_{[0,v],C_1,\ldots,C_{n-2}} \subseteq
\supp \SM_{B,C_1,\ldots,C_{n-2}}$ by \cite[Lemma 7.6.15]{Sch14}.
\end{proof}

Finally, for any $C\in\mathcal{K}^n$, the \emph{area measure} 
$\SM_{C[n-1]}$ has the representation 
\begin{equation}
\label{eq:areameas}
	\SM_{C[n-1]}(A) = \mathcal{H}^{n-1}\big(n_C^{-1}(A)\big),
\end{equation}
for $A\subseteq S^{n-1}$ \cite[\S 5.1]{Sch14},
where $\mathcal{H}^{n-1}$ denotes the $(n-1)$-Hausdorff measure and
$$
	n_C^{-1}(A) = \bigcup_{u\in A} F(C,u) \subseteq \bd C
$$
is the set of points in $\bd C$ that have a normal direction
in $A$.

\section{Proof of the upper bound}
\label{sec:pfupper}

The main aim of this section is to prove the following.

\begin{thm}
\label{thm:upper}
For any $K,L\in\mathcal{K}^n$ and $k=0,\ldots,n-2$, we have
$$
	\SM_{K[k+1],L[n-k-2]}\big(
	\big\{ u\in S^{n-1} : \dim N(K,F(K,u)) \ge n-k\big\}
	\big) = 0.
$$
\end{thm}

\smallskip

Let us first explain why this suffices to conclude the
proof of Theorem \ref{thm:mainupper}.

\begin{proof}[Proof of Theorem \ref{thm:mainupper}]
By definition,
\begin{multline*}
	\big\{ u\in S^{n-1}: u\text{ is not }(C_1,\ldots,C_{n-1})\text{-exposed}
	\big\} = \\
	\bigcup_{k=0}^{n-2}
	\bigcup_{|I|=k+1}
	\big\{ u\in S^{n-1}: \dim N(C_I,F(C_I,u)) \ge n-k \big\}.
\end{multline*}
Thus it suffices to show that
$$
	\SM_{C_1,\ldots,C_{n-1}}\big(
	\big\{ u\in S^{n-1}: \dim N(C_I,F(C_I,u)) \ge n-k \big\}
	\big)=0
$$
for all $k$ and $I\subseteq[n-1]$ with $|I|=k+1$. But as
$\SM_{C_1,\ldots,C_{n-1}}\le
\SM_{C_I[|I|],C_{I^c}[n-1-|I|]}$ by the additivity of mixed area measures,
this follows from Theorem \ref{thm:upper}.
\end{proof}

The remainder of this section is devoted to the proof of Theorem 
\ref{thm:upper}. We begin by sketching the main idea behind the proof.

\subsection{Outline of the proof}
\label{sec:upperoutline}

We begin with an elementary observation: the conclusion of Theorem 
\ref{thm:upper} is equivalent to the statement that for all $\ell>k$
$$
	\SM_{K[\ell],L[n-\ell-1]}\big(
	\big\{ u\in S^{n-1} : \dim N(K,F(K,u)) \ge n-k\big\}
	\big) = 0,
$$
since the set inside the measure increases if we replace $k$ by $\ell-1$. 
Thus the conclusion of Theorem \ref{thm:upper} is also equivalent to
$$
	\SM_{(K+tL)[n-1]}\big(
        \big\{ u\in S^{n-1} : \dim N(K,F(K,u)) \ge n-k\big\}
        \big) = O(t^{n-k-1})	
$$
as $t\downarrow 0$, since $\SM_{(K+tL)[n-1]}= \sum_{\ell=0}^{n-1}
{n-1\choose \ell} t^{n-\ell-1} \SM_{K[\ell],L[n-\ell-1]}$. We aim to
establish such a property by exploiting the representation 
\eqref{eq:areameas} 
of $\SM_{(K+tL)[n-1]}(A)$ as the area of the set of boundary points of 
$K+tL$ with a normal direction in $A$.

The basic intuition behind the proof is that the set of points in $\bd K$ 
that have a normal direction $u$ with $\dim N(K,F(K,u))\ge n-k$ is 
expected to behave as a $k$-dimensional set. Thus the corresponding 
boundary points of $K+tL$ are contained in an $L$-shaped ``tube'' with 
radius $t$ around a $k$-dimensional subset of $\bd K$, which therefore has 
area $O(t^{n-k-1})$ regardless of the choice of $L$.

That the set of $k$-singular boundary points of $K$---that is, boundary 
points with normal cone of dimension at least $n-k$---has dimension $k$
is made precise by classical results of Anderson and Klee \cite{AK52}, see 
also \cite[Theorem 2.2.5]{Sch14}: they show that this set can be covered 
by countably many compact sets of finite $k$-Hausdorff measure. This does 
not suffice for our purposes, however, since it is unclear whether this 
covering of a subset of $\bd K$ can be obtained by pulling back a covering 
of the corresponding normal directions in $S^{n-1}$. To make the argument 
work, we must first prove the existence of such a cover that behaves 
nicely under the Gauss map. We can subsequently make precise the idea that 
an $L$-shaped ``tube'' around any element of this cover has area of order 
$O(t^{n-k-1})$, concluding the proof.

\subsection{A Lipschitz cover of $k$-singular points}

The aim of this section is to prove the following result, which refines 
the results of \cite{AK52} and \cite[Theorem 2.2.5]{Sch14}.

\begin{prop}
\label{prop:klee}
Fix $K\in\mathcal{K}^n$ and $k\in \{0,\ldots,n-1\}$. Then there exists
a family of Lipschitz maps $u_m:[0,1]^k\to\bd K$ such that every face $F$ 
of $K$ with $\dim N(K,F)\ge n-k$ is contained in $\im u_m$ for some
$m\in\mathbb{N}$.
\end{prop}

We begin with some basic notations. The affine Grassmannian 
$\graff(n,k)$ is defined as the set of all 
$k$-dimensional affine subspaces of $\mathbb{R}^n$.
Denote by $\mathrm{M}_{n-k,n}$ the set of all $(n-k)\times n$ matrices 
with rank $n-k$. Then the map 
$$
	H:\mathrm{M}_{n-k,n}\times\mathbb{R}^n\to\graff(n,k)
$$
defined by
$$
	H(M,y) = \{x\in\mathbb{R}^n : M(x-y)=0\}
$$
is surjective. The following trivial fact will be used below.

\begin{fact}
\label{fact:hslice}
Consider $M\in\mathrm{M}_{n-k,n}$ so that
$M=[\,M_1\,|\,M_2\,]$ with 
$M_1\in\mathbb{R}^{(n-k)\times (n-k)}$ invertible and 
$M_2\in\mathbb{R}^{(n-k)\times k}$. Define
$g:\mathbb{R}^k\times \mathrm{M}_{n-k,n}\times \mathbb{R}^n
\to\mathbb{R}^{n-k}$ as
$$
	g(z;M,y) = M_1^{-1}(My-M_2z).
$$
Then
$$
	H(M,y) =
	\big\{ \big(g(z;M,y),z\big) : z\in \mathbb{R}^k \big\}.
$$
\end{fact}

\smallskip

In the following, we denote by $0_p$ the vector in 
$\mathbb{R}^p$ and by $0_{p\times q}$ the matrix in 
$\mathbb{R}^{p\times q}$ all of whose entries are zero, and
by $\mathbf{I}_p$ the identity matrix in $\mathbb{R}^{p\times p}$. We also 
recall 
that $p_K$ denotes the metric projection map of $K$, see 
\S\ref{sec:normalcones}.

The following is the key step in the proof of Proposition \ref{prop:klee}.

\begin{lem}
\label{lem:protoklee}
Fix $K\in\mathcal{K}^n$ and $k\in\{0,\ldots,n-1\}$, and let $S$ be a 
closed ball with $K\subset\intr S$.
For every face
$F$ of $K$ with $\dim N(K,F)=n-k$, there is a nonempty 
open set $O\subseteq \mathrm{M}_{n-k,n}\times\mathbb{R}^n$ so that
$F\subseteq p_K(H(M,y)\cap S)$ for all $(M,y)\in O$.
\end{lem}

\begin{proof}
Since the statement is invariant under rotation and translation of $K$,
we can assume without loss of generality that
$$
	N(K,F)\subseteq \mathbb{R}^{n-k}\times \{0_k\},\qquad\quad
	F=\{0_{n-k}\}\times F' \subset \{0_{n-k}\}\times \mathbb{R}^k.
$$
We will first find a single point $(M_*,y_*)$ that satisfies the 
desired conclusion, and then show it is stable under perturbation.

To construct $(M_*,y_*)$, we choose any $y_*\in\relint N(K,F)$ of
sufficiently small norm so that $F+y_*\subset\intr S$, and let
$M_*=[\,\mathbf{I}_{n-k}\, |\, 0_{(n-k)\times k}\,] \in 
\mathrm{M}_{n-k,n}$. Then 
$$
	F + y_* \subset H(M_*,y_*)\cap S,
$$
which implies that $F\subseteq p_K(H(M_*,y_*)\cap S)$ (see 
\S\ref{sec:normalcones}). 

By construction, $y_*=(y_*',0_k)$ for some $y_*'\in\mathbb{R}^{n-k}$.
Fix an open neighborhood $U'$ of $y_*'$ such that
$U=U'\times\{0_k\}\subset N(K,F)$ and $F+U \subset S$. In the following, 
we adopt the notation defined in Fact \ref{fact:hslice}. Note first that
$$
	g(z;M_*,y_*) = y_*'
$$
for all $z\in\mathbb{R}^k$. By continuity of the map $g$ and as 
$\mathbf{I}_{n-k}$ is invertible, there is an open 
neighborhood $\tilde O\subset\mathbb{R}^k\times\mathrm{M}_{n-k,n}\times
\mathbb{R}^n$ of $F'\times \{M_*\}\times\{y_*\}$ so that
$$
	g(z;M,y) \in U'\quad
	\text{and}\quad
	M_1\text{ is invertible}\quad
	\text{for all}\quad
	(z,M,y)\in \tilde O.
$$
By the generalized tube lemma \cite[\S 26, Ex.\ 9]{Mun00}, there exists an 
open neighborhood $O\subset \mathrm{M}_{n-k,n}\times\mathbb{R}^n$ 
of $(M_*,y_*)$ so that $F'\times O\subset \tilde O$.

Denote by
$n(x;M,y)=(g(x';M,y),0_k)$ for any $x=(0_{n-k},x')\in F$.
The choice of $U'$, the construction of $\tilde O$, and
Fact \ref{fact:hslice} ensure that
$$
	x + n(x;M,y) \in H(M,y)\cap S,
	\qquad\quad
	n(x;M,y) \in N(K,F)
$$
for all $(M,y)\in O$ and $x\in F$. Thus $F\subseteq p_K(H(M,y)\cap S)$
for all $(M,y)\in O$.
\end{proof}

We can now complete the proof of Proposition \ref{prop:klee}.

\begin{proof}[Proof of Proposition \ref{prop:klee}]
Fix a closed ball $S$ that contains $K$ in its interior.
For every $\ell\in\{0,\ldots,k\}$, $M\in \mathrm{M}_{n-\ell,n}$, and
$y\in\mathbb{R}^n$, the set $H(M,y)\cap S$ is either empty or a closed 
ball of dimension at most $k$. In the latter case, we can certainly find a 
Lipschitz map $f:[0,1]^k\to H(M,y)$ so that $H(M,y)\cap S\subseteq 
\mathop{\mathrm{Im}}f$.

Let $(f_m)_{m\in\mathbb{N}}$ be a family of such 
maps obtained by applying this construction to all $\ell\in\{0,\ldots,k\}$ 
and all $(M,y)$ in a countable dense subset of 
$\mathrm{M}_{n-\ell,n}\times\mathbb{R}^n$. Lemma \ref{lem:protoklee}
ensures that every face $F$ of $K$ with $\dim N(K,F)=n-\ell \ge n-k$
is contained in $\mathop{\mathrm{Im}}u_m$ for some $m$, where we define
$u_m=p_K\circ f_m$. It remains to note that $u_m$ is Lipschitz as $p_K$ is 
Lipschitz \cite[Theorem 1.2.1]{Sch14}.
\end{proof}

\subsection{Proof of Theorem \ref{thm:upper}}

Throughout the proof of Theorem \ref{thm:upper},
fix $K,L\in\mathcal{K}^n$,
$k\in \{0,\ldots,n-2\}$, and a family $(u_m)_{m\in\mathbb{N}}$ of
Lipschitz maps $u_m:[0,1]^k\to\bd K$ as in Proposition \ref{prop:klee}.
Define the sets $A_m\subseteq S^{n-1}$ as
$$
	A_m = \big\{v\in S^{n-1}: F(K,v)\subseteq\mathop{\mathrm{Im}}u_m\big\}.
$$
Then Proposition \ref{prop:klee} ensures that
$$
	\big\{ v\in S^{n-1} : \dim N(K,F(K,v)) \ge n-k\big\}
	\subseteq
	\bigcup_{m\in\mathbb{N}}A_m.	
$$
By the argument in \S\ref{sec:upperoutline}, it suffices to
show that
$$
	\SM_{(K+tL)[n-1]}(A_m) = O(t^{n-k-1})
	\quad\text{as}\quad t\downarrow 0
$$
for every $m\in\mathbb{N}$.

\begin{proof}[Proof of Theorem \ref{thm:upper}]
The representation \eqref{eq:areameas} yields
$$
	\SM_{(K+tL)[n-1]}(A_m) =
	\mathcal{H}^{n-1}\Bigg(
	\bigcup_{v\in A_m} F(K+tL,v)
	\Bigg).
$$
By the definition of $A_m$ and \eqref{eq:faceaddition}, we have
$$
	\bigcup_{v\in A_m} F(K+tL,v) \subseteq
	({\mathop{\mathrm{Im}}u_m + t L})\cap
	\bd(K+tL).
$$
Let $[0,1]^k=\bigcup_{i=1}^{N_t}B_i$ be a covering of $[0,1]^k$ 
by $N_t =O(t^{-k})$ balls $B_i$ of radius $t$. Then 
$\mathop{\mathrm{Im}}u_m \subseteq \bigcup_{i=1}^{N_t} u_m(B_i)$, and we 
can estimate
$$
	\SM_{(K+tL)[n-1]}(A_m) \le
	\sum_{i=1}^{N_t}
	\mathcal{H}^{n-1}\big(
	(u_m(B_i) + t L)\cap
        \bd(K+tL)
	\big).
$$
Note that $u_m(B_i) + t L$ is contained in a ball
$\tilde B_i$ of radius $(\mathrm{Lip}(u_m)+\mathrm{diam}(L))t$.
As $K_1\cap\bd K_2 \subseteq \bd (K_1\cap K_2)$ for any
convex bodies $K_1,K_2$, we obtain
$$
	\mathcal{H}^{n-1}\big(
	(u_m(B_i) + t L)\cap
        \bd(K+tL)
	\big) \le
	\mathcal{H}^{n-1}\big(
	{\bd\big(\tilde B_i \cap
        (K+tL)\big)}
	\big).
$$
Now recall that $\mathcal{H}^{n-1}(\bd C)=n\V_n(B,C,\ldots,C)$
for every $C\in\mathcal{K}^n$ \cite[(5.53)]{Sch14}, so
$\mathcal{H}^{n-1}(\bd K_1)\le\mathcal{H}^{n-1}(\bd K_2)$ for any convex 
bodies $K_1\subseteq K_2$. Thus
$$
	\mathcal{H}^{n-1}\big(
	{\bd\big(\tilde B_i \cap
        (K+tL)\big)
	}\big)
	\le	
	\mathcal{H}^{n-1}\big(
	{\bd \tilde B_i}
	\big)
	=
	c t^{n-1}
$$
with $c=(\mathrm{Lip}(u_m)+\mathrm{diam}(L))^{n-1}
\mathcal{H}^{n-1}(S^{n-1})$.
We therefore obtain 
$$
	\SM_{(K+tL)[n-1]}(A_m)
	\le 
	ct^{n-1}N_t 
	= O(t^{n-k-1}),
$$
concluding the proof.
\end{proof}

\section{Proof of the lower bound}
\label{sec:pflower}

The aim of this section is to prove Theorem \ref{thm:mainlower}.
We again begin by sketching the main idea behind the proof before 
proceeding to the details.

\subsection{Outline of the proof}
\label{sec:outlinelower}

We aim to characterize $\supp \SM_{K[n-2],L}$ for $K,L\in\mathcal{K}^n$ by 
induction on the dimension $n$. The basic principle for doing so was 
explained in \S\ref{sec:higher}, but the procedure cannot be implemented 
directly as \eqref{eq:ifonly} does not always hold. Instead, we will 
perform the induction in an indirect manner that requires only the case 
$\dim T(K,u)=1$ of \eqref{eq:ifonly}, which holds by Theorem 
\ref{thm:mainproj}.

\begin{rem}
In the following, we will take for granted that Conjecture 
\ref{conj:schneider} is known to hold for the area measure $\SM_{K[n-1]}$,
that is, that
$$
	\supp \SM_{K[n-1]} =
	\cl\big\{u\in S^{n-1}:\dim T(K,u)=1\big\}.
$$
This is a well known result, see, e.g., \cite[Lemma 4.5.2]{Sch14};
we include a short proof in Corollary 
\ref{cor:areasupp} below that follows trivially from the methods used in 
the proof of Theorem \ref{thm:mainproj}.
Note that this implies, in particular, that $\supp \SM_{K[n-2],L}$ is 
characterized for $n=2$, which serves as the base case of our induction
on $n$.
\end{rem}

Let us now sketch the idea behind the induction. By Corollary 
\ref{cor:mainupper} and as the support of a measure is a
closed set, it suffices to show that any $(K[n-2],L)$-extreme 
direction $u\in S^{n-1}$ is contained in $\supp \SM_{K[n-2],L}$.
We fix such a direction $u$ in the following. Let us consider two 
nearly complementary cases.
\smallskip
\begin{enumerate}[1.]
\itemsep\medskipamount
\item If $\dim T(K,u)=1$, then the approach described in 
\S\ref{sec:higher} can be applied directly: Lemma \ref{lem:extremeproj}
yields $v\in u^\perp$ so that $u$ is
$(\proj_{v^\perp}K[n-3],\proj_{v^\perp}L)$-extreme, and thus the induction 
hypothesis implies that $u\in\supp 
\SM_{\proj_{v^\perp}K[n-3],\proj_{v^\perp}L}$. We now conclude that
$u\in \supp \SM_{K[n-2],L}$ by the case $\dim T(K,u)=1$ of Theorem 
\ref{thm:mainproj}.
\item Now suppose that $\dim T(K,u')>1$ for all $u'$ in a neighborhood
of $u$. Then we have $u\not\in \supp \SM_{K[n-1]}$. It therefore suffices to 
show that 
$$
	u\in \supp\big(\SM_{K[n-2],L}+\SM_{K[n-1]}\big) = \supp \SM_{K[n-2],K+L}.
$$
To this end, note that it is easily checked using Definition 
\ref{defn:extreme} that the fact that $u$ is
$(K[n-2],L)$-extreme implies that it is also
$(K[n-2],K+L)$-extreme, as well as that $\dim T(K+L,u)=1$. We can 
therefore once again proceed as in \S\ref{sec:higher} 
to achieve the desired conclusion using Theorem \ref{thm:mainproj}.
\end{enumerate}
\smallskip
These two cases almost, but not quite, suffice to prove Theorem 
\ref{thm:mainlower}: the above arguments leave the case that
$\dim T(K,u)>1$, but $u$ lies on the boundary of the set of
$u'$ with $\dim T(K,u')=1$, unresolved.

To handle this boundary case, we will use to our advantage the fact that 
the measure $\SM_{C_1,\ldots,C_{n-1}}$ does \emph{not} uniquely determine 
the convex bodies $C_1,\ldots,C_{n-1}$. We will show in 
\S\ref{sec:lowermainpf} how one can modify the body $K$ so that it 
satisfies the condition of the second case above, without changing the 
mixed area measure $\SM_{K[n-2],L}$. Then the second case suffices to 
complete the proof of Theorem \ref{thm:mainlower} (and, somewhat 
surprisingly, the first case no longer needs to be considered separately).

\subsection{The case $\dim T(K,u)=1$ of Theorem \ref{thm:mainproj}}
\label{sec:projcase1}

Before we proceed to the proof of Theorem \ref{thm:mainlower}, we must 
prove the case $\dim T(K,u)=1$ of Theorem \ref{thm:mainproj} that is 
needed here. The remaining parts of Theorem \ref{thm:mainproj} will be 
proved in \S\ref{sec:obstruction}.

The proof relies on a characterization of 
directions $u$ with $\dim T(K,u)=1$, which is dual to the 
characterization of extreme boundary points of a convex
body in \cite[Lemma 1.4.6]{Sch14}. Only the implication 
$a\Rightarrow b$ will be used in the sequel.

\begin{lem}
\label{lem:0extcap}
For any $K\in\mathcal{K}^n$ and $u\in S^{n-1}$, the following are 
equivalent.
\smallskip
\begin{enumerate}[a.]
\itemsep\smallskipamount
\item $\dim T(K,u)=1$.
\item For every $\varepsilon>0$, there exists $K'\in\mathcal{K}^n_n$ so 
that $K\subset K'$, $h_{K'}(u)>h_K(u)$, and $h_{K'}(v)=h_K(v)$ for all
$v\in S^{n-1}\backslash B(u,\varepsilon)$.
\end{enumerate}
\end{lem}

\begin{proof}
We first prove $a\Rightarrow b$. As the statement is invariant 
under translation, we assume without loss of generality that 
$\relint K$ contains the origin.

If $\dim K\le n-2$, then $\dim T(K,u)>1$ for all $u$ and there is nothing 
to prove. If $\dim K=n-1$, then $\dim T(K,u)=1$ implies that $u\in 
K^\perp$. Let $K'=\mathrm{conv}\{K,tu\}$ with $t>0$. Then 
$h_{K'}(u)=t>0=h_K(u)$. Moreover, for every $\varepsilon>0$, we can choose 
$t>0$ so that $h_{K'}=h_K$ on $S^{n-1}\backslash 
B(u,\varepsilon)$, since as $t\downarrow 0$ the normal direction of each 
supporting hyperplane of $K'$ that does not touch $K$ converges to $u$.

It remains to consider the case that $\dim K=n$, so that 
$K\in\mathcal{K}^n_{(o)}$.
By the duality explained in \S\ref{sec:touchingcones}, $\dim T(K,u)=1$
implies that 
$$
	T(K,u)\cap \bd K^\circ = \{x\}
	\quad\text{with}\quad
	x=\frac{u}{\|u\|_{K^\circ}}
$$
is a $0$-dimensional face of $K^\circ$. By \cite[Lemma 1.4.6]{Sch14},
there exists for every $\varepsilon>0$ a choice of
$v\in S^{n-1}$ and $t\in\mathbb{R}_+$ so that $x\in\intr H^+_{v,t}$ and
$$
	K^\circ \cap H^+_{v,t} \subset \mathbb{R}_+B(u,\varepsilon).
$$
Now define $K' = (K^\circ \cap H^-_{v,t})^\circ \supset K$. Then
$$
	h_{K'}(v) = \|v\|_{K^\circ \cap H^-_{v,t}}
	= \|v\|_{K^\circ} = h_K(v)
$$
for all $v\in S^{n-1}\backslash B(u,\varepsilon)$. Moreover, as
$x\not\in K^\circ \cap H^-_{v,t}$, we have
$$
	h_{K'}(u)
	= \|u\|_{K^\circ \cap H^-_{v,t}} > \|u\|_{K^\circ} =
	h_K(u),
$$
which concludes the proof of $a\Rightarrow b$.

To prove the converse implication $b\Rightarrow a$, suppose that
$\dim T(K,u)>1$. Since $u\in \relint T(K,u)$, we can find
distinct $v_1,v_2\in T(K,u)$ so that $u = \frac{v_1+v_2}{2}$.
Now assume for sake of contradiction that the conclusion of part $b$ 
holds, and choose $\varepsilon>0$ sufficiently small that
$v_1,v_2\not\in B(u,\varepsilon)$. Then
$$
	h_K(u) = \frac{h_K(v_1)+h_K(v_2)}{2} =
	\frac{h_{K'}(v_1)+h_{K'}(v_2)}{2}
	\ge h_{K'}(u),
$$
where we used that $h_K$ is affine on $T(K,u)$ in the first equality, and 
that $h_{K'}$ is 
convex in the inequality. This entails a contradiction, concluding the 
proof.
\end{proof}

Recall the following classical result 
\cite[Theorems 7.3.1 and 7.4.2]{Sch14}.

\begin{thm}[Alexandrov--Fenchel inequality]
\label{thm:af}
For any $K,L,C_1,\ldots,C_{n-2}\in\mathcal{K}^n$
$$
	\V_n(K,L,C_1,\ldots,C_{n-2})^2 \ge
	\V_n(K,K,C_1,\ldots,C_{n-2})\,
	\V_n(L,L,C_1,\ldots,C_{n-2}).
$$
If in addition $\V_n(K,L,C_1,\ldots,C_{n-2})>0$, then equality
holds if and only if
$$
	\SM_{K,C_1,\ldots,C_{n-2}} = 
	\frac{\V_n(K,K,C_1,\ldots,C_{n-2})}{\V_n(K,L,C_1,\ldots,C_{n-2})}\,
	\SM_{L,C_1,\ldots,C_{n-2}}.
$$
\end{thm}

\smallskip

The following is the main observation of this section.

\begin{lem}
\label{lem:0extaf}
Let $K,C_1,\ldots,C_{n-2}\in\mathcal{K}^n$ and $u\in S^{n-1}$ with $\dim 
T(K,u)=1$. 
If $u\not\in\supp \SM_{K,C_1,\ldots,C_{n-2}}$, then
$u\not\in\supp \SM_{L,C_1,\ldots,C_{n-2}}$ for every $L\in\mathcal{K}^n$.
\end{lem}

\begin{proof}
As we assume that $u\not\in\supp \SM_{K,C_1,\ldots,C_{n-2}}$, 
there exists $\varepsilon>0$ so that 
$\SM_{K,C_1,\ldots,C_{n-2}}(B(u,\varepsilon))=0$.
Let $K'\supset K$ be the convex body that is provided by part 
$b$ of Lemma \ref{lem:0extcap} for this choice of $\varepsilon$.

Suppose that $\V_n(K',K,C_1,\ldots,C_{n-2})=0$. Recall that
$\dim K'=n$ by construction, while $\dim K\ge n-1$ as $\dim 
T(K,u)=1$. Thus by Fact~\ref{fact:mvpost}, there exists $I\subseteq[n-2]$
so that $\dim C_I < |I|$. Then Fact~\ref{fact:mvpost} readily yields
that $\SM_{L,C_1,\ldots,C_{n-2}}=0$ for every $L\in\mathcal{K}^n$, 
concluding the proof in this case.

We may therefore assume that $\V_n(K',K,C_1,\ldots,C_{n-2})>0$.
Note that by construction $h_{K'}=h_K$ on
$S^{n-1}\backslash B(u,\varepsilon)
\supseteq \supp \SM_{K,C_1,\ldots,C_{n-2}}$, so that
$$
	\V_n(K',K,C_1,\ldots,C_{n-2}) = 
	\V_n(K,K,C_1,\ldots,C_{n-2}).
$$
Therefore
\begin{align*}
	\V_n(K',K,C_1,\ldots,C_{n-2})^2 &=
	\V_n(K',K,C_1,\ldots,C_{n-2}) \,\V_n(K,K,C_1,\ldots,C_{n-2}) \\
	&\le
	\V_n(K',K',C_1,\ldots,C_{n-2}) \,\V_n(K,K,C_1,\ldots,C_{n-2}),
\end{align*}
where we used that $K\subseteq K'$. Thus Theorem \ref{thm:af} yields
$\SM_{h_{K'}-h_K,C_1,\ldots,C_{n-2}}=0$. In particular, integrating the 
function $h_L$ with respect to this measure yields
$$
	0 =
	\V_{n}(L,h_{K'}-h_K,C_1,\ldots,C_{n-2}) =
	\frac{1}{n}\int (h_{K'}-h_K)\,
	d\SM_{L,C_1,\ldots,C_{n-2}}
$$
for every $L\in\mathcal{K}^n$. Since the function $h_{K'}-h_K$ is 
nonnegative on $S^{n-1}$ and is strictly 
positive in a neighborhood of $u$, we conclude that
$u\not\in \supp \SM_{L,C_1,\ldots,C_{n-2}}$.
\end{proof}

The following follows directly.

\begin{proof}[Proof of the case $\dim T(K,u)=1$ of
Theorem \ref{thm:mainproj}]
The conclusion is immediate by choosing $L=[0,v]$ in
Lemma \ref{lem:0extaf} and applying the projection formula
in \S\ref{sec:mvma}.
\end{proof}

As a further simple illustration of the utility of Lemma \ref{lem:0extaf}, 
let us give another proof of the following well known result (see,
e.g., \cite[\S 4.5]{Sch14}).

\begin{cor}
\label{cor:areasupp}
For any $K\in\mathcal{K}^n$, we have
$$
	\supp \SM_{K[n-1]} =
	\cl\big\{u\in S^{n-1}:\dim T(K,u)=1\big\}.
$$
\end{cor}

\begin{proof}
As the upper bound holds by Corollary \ref{cor:mainupper}, it suffices to 
show that every $u\in S^{n-1}$ with $\dim T(K,u)=1$ satisfies
$u\in \supp \SM_{K[n-1]}$. Suppose the latter does not hold. Then we may 
repeatedly apply Lemma \ref{lem:0extaf} with $L=B$ to conclude that
$u\not\in \SM_{B[n-1]}$. But this entails a contradiction, as $\SM_{B[n-1]}$ 
is the Lebesgue measure on 
$S^{n-1}$ by \eqref{eq:areameas} and thus $\supp \SM_{B[n-1]}=S^{n-1}$.
\end{proof}

\subsection{Proof of Theorem \ref{thm:mainlower}}
\label{sec:lowermainpf}

As was explained in \S\ref{sec:outlinelower}, the main difficulty in the 
proof is to modify $K$ so that it has the desired properties without 
changing the mixed area measure. We begin with a useful observation.

\begin{lem}
\label{lem:wulff}
Let $K\in\mathcal{K}^n$, $E\subseteq S^{n-1}$, and
$f:E\to\mathbb{R}$ such that
$$
	K = \bigcap_{v\in E} H_{v,f(v)}^-.
$$
Then every $u\in S^{n-1}$ with $\dim T(K,u)=1$ is in the closure
of $E$.
\end{lem}

\begin{proof}
Let $u\in S^{n-1}$ with $\dim T(K,u)=1$, and suppose that $u\not\in\cl E$.
Choose $\varepsilon>0$ so that $B(u,\varepsilon)\cap E = \varnothing$,
and let $K'\supset K$ be the resulting convex body that is provided by 
part $b$ of Lemma \ref{lem:0extcap}. Then $h_{K'}(v)=h_K(v)$ for all $v\in 
E$ and thus
$$
	K' \subseteq \bigcap_{v\in E} H_{v,h_{K'}(v)}^- =
	\bigcap_{v\in E} H_{v,h_{K}(v)}^-
	\subseteq \bigcap_{v\in E} H_{v,f(v)}^- = K,
$$
where we used the obvious fact that $h_K(v)\le f(v)$ for every $v\in 
S^{n-1}$. This entails a contradiction, since $h_{K'}(u)>h_K(u)$.
\end{proof}

We can now construct the desired modification of $K$.

\begin{lem}
\label{lem:maxbody}
Let $K,C_1,\ldots,C_{n-2}\in\mathcal{K}^n$, $u\in S^{n-1}$, and 
$\varepsilon>0$ so that $B(u,\varepsilon)$ lies strictly
inside a hemisphere.
Suppose that $\SM_{K,C_1,\ldots,C_{n-2}}(B(u,\varepsilon))=0$. Then
$$
	\hat K = \bigcap_{v\in S^{n-1}\backslash B(u,\varepsilon)}
	H^-_{v,h_K(v)} \supseteq K
$$
is a convex body that satisfies the following properties.
\smallskip
\begin{enumerate}[a.]
\itemsep\smallskipamount
\item $\SM_{\hat K,C_1,\ldots,C_{n-2}}=\SM_{K,C_1,\ldots,C_{n-2}}$.
\item $\dim T(\hat K,v)>1$ for every $v\in B(u,\varepsilon)$.
\item If in addition $u\in \supp \SM_{B,C_1,\ldots,C_{n-2}}$, then
$T(\hat K,u)\subseteq T(K,u)$.
\end{enumerate}
\end{lem}

\begin{proof}
We first note that as $B(u,\varepsilon)$ lies strictly inside a 
hemisphere, it follows that $\hat K\subseteq \bigcap_{v\in S^{n-1}\backslash 
B(u,\varepsilon)} H^-_{v,\|h_K\|_\infty}$ is compact and thus
$\hat K\in\mathcal{K}^n$.

We now prove part $a$. Suppose first that $\dim K=n$. Then we can proceed 
as in the proof of Lemma \ref{lem:0extaf}. Indeed, if $\V_n(\hat 
K,K,C_1,\ldots,C_{n-2})=0$, then Fact \ref{fact:mvpost} yields 
$I\subseteq[n-2]$ with $\dim C_I<|I|$, and thus $\SM_{\hat 
K,C_1,\ldots,C_{n-2}}=\SM_{K,C_1,\ldots,C_{n-2}}=0$. If $\V_n(\hat 
K,K,C_1,\ldots,C_{n-2})>0$, then the Alexandrov--Fenchel argument of Lemma 
\ref{lem:0extaf} extends \emph{verbatim} to the present setting as 
$h_{\hat K}(v)=h_K(v)$ for all $v\in S^{n-1}\backslash B(u,\varepsilon)$.
This concludes the proof of part $a$ for $\dim K=n$.

If $\dim K<n$, define $K_t=K+t\hat B$, where $\hat 
B=\mathrm{conv}\{B,su\}$ and $s>1$ is chosen so that $N(\hat 
B,su)=\mathbb{R}_+\cl B(u,\varepsilon)$. Then $\SM_{\hat 
B,C_1,\ldots,C_{n-2}}(B(u,\varepsilon))=0$ by Theorem \ref{thm:mainupper},
so we can apply part $a$ to $K_t$ to conclude that
$\SM_{\hat K_t,C_1,\ldots,C_{n-2}}=\SM_{K_t,C_1,\ldots,C_{n-2}}$.
The conclusion of part $a$ follows by the continuity
of mixed area measures \cite[p.\ 281]{Sch14}, as
$K_t\to K$ and $\hat K_t\to\hat K$ as $t\downarrow 0$ \cite[Lemma 
7.5.2]{Sch14}.

Part $b$ follows directly from the definition of $\hat K$ by
Lemma \ref{lem:wulff} and as $S^{n-1}\backslash B(u,\varepsilon)$ is
closed (as $B(u,\varepsilon)$ is an open ball by definition).

To prove part $c$, note that integrating both sides of part $a$ with
respect to $h_B$, and using that mixed volumes are symmetric in their 
arguments, yields
$$
	\int (h_{\hat K}-h_K)\, d\SM_{B,C_1,\ldots,C_{n-2}} =
	\int h_B\, d\SM_{h_{\hat K}-h_K,C_1,\ldots,C_{n-2}} = 0.
$$
Since $h_{\hat K}-h_K\ge 0$ and $u\in \supp \SM_{B,C_1,\ldots,C_{n-2}}$,
this implies that $h_{\hat K}(u)=h_K(u)$. As $K\subseteq\hat K$, 
it follows that $F(K,u)\subseteq F(\hat K,u)$. 
Now fix any $x\in F(K,u)$.
Since $x\in\bd K\cap \bd \hat K$
and $K\subseteq\hat K$, any supporting hyperplane of $\hat K$ that 
contains 
$x$ must also be a supporting hyperplane of $K$, so we must have $u\in 
N(\hat K,x)\subseteq N(K,x)$.
The conclusion of part $c$ now follows from
Lemma \ref{lem:touchincl}.
\end{proof}

We can now conclude the proof of Theorem \ref{thm:mainlower}.

\begin{proof}[Proof of Theorem \ref{thm:mainlower}]
Let $K,L\in\mathcal{K}^n$. As the upper bound is provided by 
Corollary~\ref{cor:mainupper}, it suffices to show that every $(K[n-2],L)$-extreme
direction $u\in S^{n-1}$ is contained in $\supp \SM_{K[n-2],L}$.
We will do so by induction on the dimension $n$. 
The base case $n=2$ is provided by Corollary \ref{cor:areasupp}; we
assume in the remainder of the proof that the result has been proved up to 
dimension $n-1$.

Fix a $(K[n-2],L)$-extreme direction $u\in 
S^{n-1}$, and assume for sake of contradiction that
$u\not\in \supp \SM_{K[n-2],L}$. Then there exists $\varepsilon>0$
sufficiently small so that $\SM_{K[n-2],L}(B(u,\varepsilon))=0$.
Let $\hat K$ be the modified body of Lemma \ref{lem:maxbody}. Then 
applying part $a$ of Lemma \ref{lem:maxbody} repeatedly yields
$\SM_{\hat K[n-2],L}=\SM_{K[n-2],L}$. In particular,
$$
	u\not\in \supp \SM_{\hat K[n-2],L}.
$$
We must now establish the following.

\begin{claim}
The vector $u$ is $(\hat K[n-2],L)$-extreme.
\end{claim}

\begin{proof}
As $u$ is $(K[n-2],L)$-extreme, Lemma \ref{lem:extremeproj} yields
$v\in T(K,u)^\perp$ so that
$u$ is $(\proj_{v^\perp}K[n-3],\proj_{v^\perp}L)$-extreme.
Thus $u\in \supp \SM_{\proj_{v^\perp}K[n-3],\proj_{v^\perp}L}$
by the induction hypothesis, which implies
$u\in\supp \SM_{B,K[n-3],L}$
by Lemma \ref{lem:mixedareaproj}. Part $c$ of Lemma~\ref{lem:maxbody}
therefore yields $T(\hat K,u)\subseteq T(K,u)$, which readily implies
the claim.
\end{proof}

We now proceed as was explained in \S\ref{sec:outlinelower}.
We have
$$
	u\not\in \supp \SM_{\hat K[n-1]}
$$
by
Corollary \ref{cor:areasupp} and part $b$ of Lemma \ref{lem:maxbody}.
Therefore
$$
	u\not\in \supp \big(\SM_{\hat K[n-2],L} + \SM_{\hat K[n-1]}\big) =
	\supp \SM_{\hat K[n-2],\hat K+L}.
$$
On the other hand, as $u$ is $(\hat K[n-2],L)$-extreme, $u$ is also
$(\hat K[n-2],\hat K+L)$-extreme and $\dim T(\hat K+L,u)=1$.
By Lemma \ref{lem:extremeproj}, there exists $v\in u^\perp$ so that
$u$ is $(\proj_{v^\perp}\hat K[n-2])$-extreme, and thus
$u\in\supp \SM_{\proj_{v^\perp}\hat K[n-2]}$ by Corollary 
\ref{cor:areasupp}. The case $\dim T(K,u)=1$ of Theorem \ref{thm:mainproj}
now yields $u\in \supp \SM_{\hat K[n-2],\hat K+L}$.
This entails a contradiction, concluding the proof.
\end{proof}

\section{Mixed Hessian measures}
\label{sec:mixedhess}

The aim of this section is to derive the implications of our main results 
for mixed Hessian measures. In particular, we will prove 
Corollaries \ref{cor:mixedhess} and \ref{cor:mixedhnnonsm}.

It should be emphasized that mixed area and Hessian measures are not 
really distinct notions: mixed areas measures of convex bodies in 
$\mathbb{R}^n$ are a special case of mixed Hessian measures of convex 
functions on $\mathbb{R}^n$ \cite[Corollary 4.2]{HMU24}, while mixed 
Hessian measures of convex functions on $\mathbb{R}^n$ are a special case 
of mixed area measures of convex bodies in $\mathbb{R}^{n+1}$ 
\cite[Corollary 4.9]{HMU24}. Such connections were used in 
\cite{CH05,HMU24} to translate some previously known special cases of 
Conjecture~\ref{conj:schneider} to the setting of mixed Hessian measures. 
Some care is needed, however, in extending these connections to the 
general setting considered here.

\subsection{From mixed area to mixed Hessian measures}

The aim of this section is to recall how mixed Hessian measures
can be obtained as a special case of mixed area measures. The construction 
in this section is taken from \cite{HMU24}.

Denote by $\mathrm{Conv}_{\rm cd}(\mathbb{R}^n)$ the class of 
lower-semicontinuous proper convex functions $g:\mathbb{R}^n\to 
(-\infty,\infty]$ such that $\mathop{\mathrm{dom}} g$ is compact, and let
$$
	\mathrm{Conv}_{\rm cd}^*(\mathbb{R}^n) =
	\{ g^* : g\in \mathrm{Conv}_{\rm cd}(\mathbb{R}^n)\}.
$$
Here $g^*$ denotes the convex conjugate
$$
	g^*(x) = \sup_{y\in\mathbb{R}^n}\{\langle x,y\rangle -
	g(y)\}.
$$
Let $f\in \mathrm{Conv}_{\rm cd}^*(\mathbb{R}^n)$.
Then $f^*\in\mathrm{Conv}_{\rm cd}(\mathbb{R}^n)$, as
any lower-semicontinuous convex function $g$
satisfies $g^{**}=g$ \cite[Theorem 1.6.13]{Sch14}.
As $f^*$ is lower-semicontinuous,
$$
	\mathrm{epi}(f^*) =
	\{(x,t)\in\mathbb{R}^{n+1}:
	f^*(x)\le t\}
$$
is a closed convex set, where we write $(x,t)\in\mathbb{R}^{n+1}$ with
$x\in\mathbb{R}^n$ and $t\in\mathbb{R}$.

\begin{defn}
For every $f\in \mathrm{Conv}_{\rm cd}^*(\mathbb{R}^n)$, define
$K_f\in\mathcal{K}^{n+1}$ as
$$
	K_f = \mathrm{epi}(f^*) \cap \{(x,t)\in\mathbb{R}^{n+1}:
	t\le r_f\}
$$
where $r_f=\max_{x\in\mathop{\mathrm{dom}}f^*} f^*(x)<\infty$.
\end{defn}

In the following, we denote by $S^n_-$ the negative hemisphere in $S^n$:
$$
	S^n_- = \{(y,t)\in S^n: t<0\}.
$$
Then the map $\phi:\mathbb{R}^n\to S^n_-$ defined by
$$
	\phi(x) = \bigg(
	\frac{x}{\sqrt{1+\|x\|^2}},
	-\frac{1}{\sqrt{1+\|x\|^2}}
	\bigg)
$$
is a homeomorphism with inverse $\phi^{-1}(y,t)=-\frac{y}{t}$. The 
following result, which is a restatement of \cite[Corollary 4.9]{HMU24},
relates the pushforward
$\phi_*\HH_{f_1,\ldots,f_n}$ of the mixed Hessian measure to the
mixed area measure $\SM_{K_{f_1},\ldots,K_{f_n}}$ restricted
to $S^n_-$.

\begin{lem}
\label{lem:hmu}
Let $f_1,\ldots,f_n\in \mathrm{Conv}_{\rm cd}^*(\mathbb{R}^n)$.
Then the measures $\phi_*\HH_{f_1,\ldots,f_n}$
and $\SM_{K_{f_1},\ldots,K_{f_n}}|_{S^n_-}$ are mutually absolutely 
continuous with
$$
	\frac{d\phi_*\HH_{f_1,\ldots,f_n}}{d\SM_{K_{f_1},\ldots,K_{f_n}}}
	(\phi(x)) = \frac{1}{\sqrt{1+\|x\|^2}}.
$$
In particular,
$\supp \HH_{f_1,\ldots,f_n} = \phi^{-1}\big(S^n_- \cap
\supp \SM_{K_{f_1},\ldots,K_{f_n}}\big)$.
\end{lem}

\subsection{Extreme directions and localization}

In order to translate our main results to the setting of mixed Hessian
measures, we must clarify the connection between 
$L(f,x)$ and $T(K_f,u)$. We begin with the following simple lemma.

\begin{lem}
\label{lem:kfsupp}
For any $f\in\mathrm{Conv}_{\rm cd}^*(\mathbb{R}^n)$ and
$x\in\mathbb{R}^n$, we have
$f(x) = h_{K_f}(x,-1)$.
\end{lem}

\begin{proof}
We readily compute
$$
	h_{K_f}(x,-1) =
	\sup_{(y,t): f^*(y)\le t\le r_f}
	\{\langle y,x\rangle - t\}
	=
	\sup_y \{\langle y,x\rangle - f^*(y)\} = f(x)
$$
by the definition of $K_f$.
\end{proof}

Consequently, we have the following.

\begin{lem}
\label{lem:lftkf}
For any $f\in\mathrm{Conv}_{\rm cd}^*(\mathbb{R}^n)$ and
$x\in\mathbb{R}^n$, we have
$$
	L(f,x) = \phi^{-1}\big(
	S^n_- \cap T(K_f,\phi(x))
	\big).
$$
\end{lem}

\begin{proof}
Any convex set $B\subseteq\mathbb{R}^n$ defines a convex cone
$A=\mathbb{R}_+\phi(B)$ in $\mathbb{R}_+S^n_-$.
We claim that $f$ is affine on $B$ if and only if 
$h_{K_f}$ is linear on $A$: indeed,
\begin{align*}
	&f\bigg(\frac{t}{t+t'}x+\frac{t'}{t+t'}x'\bigg) =
	\frac{t}{t+t'} f(x) +
	\frac{t'}{t+t'} f(x') \quad
	\Longleftrightarrow
\\
	&h_{K_f}(tx+t'x',-t-t') =
	h_{K_f}(tx,-t) +
	h_{K_f}(t'x',-t') 
\end{align*}
for any $x,x'\in B$ and
$t,t'>0$ by Lemma \ref{lem:kfsupp}. 

Now note that for any convex cone $A$ in $\mathbb{R}_+S^n_-$, the convex 
set $B\subseteq\mathbb{R}^n$ defined by $B=\phi^{-1}(A\cap S^n_-)$ 
satisfies $A=\mathbb{R}_+\phi(B)$. Thus the relation between $A$ and $B$ 
defines a bijection between convex sets in $\mathbb{R}^n$ and convex cones 
in $\mathbb{R}_+S^n_-$. The definition of $L(f,x)$ therefore implies that
$\mathbb{R}_+\phi(L(f,x))$ is the largest convex cone in
$\mathbb{R}_+S^n_-$ that contains $\phi(x)$ in its relative interior on 
which $h_{K_f}$ is linear, that is,
$$
	\mathbb{R}_+\phi(L(f,x)) = T(K_f,\phi(x)) \cap \mathbb{R}_+S^n_-.
$$
The conclusion follows by inverting this identity.
\end{proof}

Lemma \ref{lem:lftkf} immediately yields the following corollary.

\begin{cor}
\label{cor:hessext}
Let $f_1,\ldots,f_n\in\mathrm{Conv}_{\rm cd}^*(\mathbb{R}^n)$.
Then $x\in\mathbb{R}^n$ is $(f_1,\ldots,f_n)$-extreme if and only if
$\phi(x)\in S^n$ is $(K_{f_1},\ldots,K_{f_n})$-extreme.
\end{cor}

We are now nearly ready to prove our main results on mixed Hessian 
measures. There is, however, one additional technical point that must be 
dispensed with first. The reduction from mixed Hessian measures to mixed 
area measures requires us to work with functions in $\mathrm{Conv}_{\rm 
cd}^*(\mathbb{R}^n)$, while our results are formulated in the more general 
setting of functions in $\mathrm{Conv}(\Omega)$ for some open convex set 
$\Omega\subseteq\mathbb{R}^n$. The following localization lemma reduces 
the the latter setting to the former one.

\begin{lem}
\label{lem:loc}
Let $\Omega\subseteq\mathbb{R}^n$ be an open convex set
and $\Omega'\subset\Omega$ be convex and compact.
Then for any $f\in \mathrm{Conv}(\Omega)$, there exists $g\in 
\mathrm{Conv}_{\rm cd}^*(\mathbb{R}^n)$ so that $f=g$ on $\Omega'$.
\end{lem}

\begin{proof}
Recall that the indicator $\mathrm{I}_A$ of a convex set $A$ is defined 
by
$$
	\mathrm{I}_A(x) = \begin{cases}
	0 & \text{for }x\in A,\\
	+\infty & \text{for }x\not\in A.
	\end{cases}
$$
By \cite[Theorem 1.5.3]{Sch14}, $f$ is Lipschitz on $\Omega'$ with 
Lipschitz constant $\ell$. Define
$$
	g = ((f + \mathrm{I}_{\Omega'})^*+\mathrm{I}_{\ell B})^*.
$$
As $\mathrm{dom}(f + \mathrm{I}_{\Omega'})=\Omega'$ is compact,
$(f + \mathrm{I}_{\Omega'})^*$ is a continuous convex function on
$\mathbb{R}^n$. Thus $(f + \mathrm{I}_{\Omega'})^*+\mathrm{I}_{\ell B}
\in \mathrm{Conv}_{\rm cd}(\mathbb{R}^n)$ and therefore
$g\in \mathrm{Conv}_{\rm cd}^*(\mathbb{R}^n)$.

Now note that $f + \mathrm{I}_{\Omega'}$ is lower-semicontinuous, and
thus
$$
	f(x) =
	\sup_{y\in\mathbb{R}^n}
	\big\{ \langle y,x\rangle - (f + \mathrm{I}_{\Omega'})^*(y)\big\}
$$
for $x\in\Omega'$.
By \cite[Theorem 23.5]{Roc70}, 
the supremum is attained by $y\in \partial f(x)\subseteq
\partial(f + \mathrm{I}_{\Omega'})(x)$, where 
$\partial f$ is the subdifferential.
As $\|y\|\le\ell$ by \cite[Theorem 24.7]{Roc70}, 
$$
	f(x) =
	\sup_{\|y\|\le\ell}
	\big\{ \langle y,x\rangle - (f + \mathrm{I}_{\Omega'})^*(y)\big\}	
	= g(x)
$$
for every $x\in\Omega'$, concluding the proof.
\end{proof}

Fix $f_1,\ldots,f_n\in\mathrm{Conv}(\Omega)$ and $x\in\Omega$,
and let $\Omega'=\cl B(x,\varepsilon)$ for some $\varepsilon>0$.
Construct the functions $g_1,\ldots,g_n\in  \mathrm{Conv}_{\rm 
cd}^*(\mathbb{R}^n)$ as in Lemma \ref{lem:loc}. Then clearly
$$
	\HH_{f_1,\ldots,f_n}|_{B(x,\varepsilon)} =
	\HH_{g_1,\ldots,g_n}|_{B(x,\varepsilon)},
	\qquad
	T(f_i,x)\cap B(x,\varepsilon) = T(g_i,x)\cap B(x,\varepsilon).
$$
In particular, the properties that $u\in \supp \HH_{f_1,\ldots,f_n}$ or that
$u$ is $(f_1,\ldots,f_n)$-extreme are unchanged if we replace
$f_1,\ldots,f_n$ by $g_1,\ldots,g_n$.

\begin{cor}
\label{cor:equivconj}
Conjecture \ref{conj:schneider} implies Conjecture 
\ref{conj:fcnschneider}.
\end{cor}

\begin{proof}
Since the statements of both conjectures are local in nature, it
suffices by the remark following Lemma \ref{lem:loc} to assume that
$f_1,\ldots,f_n\in \mathrm{Conv}_{\rm cd}^*(\mathbb{R}^n)$.
The conclusion now follows directly from Lemma \ref{lem:hmu}
and Corollary \ref{cor:hessext}.
\end{proof}

\begin{rem}
Conjectures \ref{conj:schneider} and \ref{conj:fcnschneider} are
in fact equivalent; the converse implication to Corollary 
\ref{cor:equivconj} can be readily read off from \cite[Corollary 
4.2]{HMU24}.
\end{rem}

The proof of Corollary \ref{cor:mixedhess} is completely analogous.

\begin{proof}[Proof of Corollary \ref{cor:mixedhess}]
The result follows from Theorems \ref{thm:mainupper} and 
\ref{thm:mainlower} by exactly the same argument as in the proof of 
Corollary \ref{cor:equivconj}.
\end{proof}

We conclude by proving Corollary \ref{cor:mixedhnnonsm}. In the proof, we 
will use without comment that Lemma \ref{lem:touchingface} extends 
directly to the convex function setting.

\begin{proof}[Proof of Corollary \ref{cor:mixedhnnonsm}]
As $\HH_{f,g}(D)=0$, Corollary \ref{cor:mixedhnnonsm} shows that
none of the points in $D$ are $(f,g)$-extreme. Thus for any $x\in 
D\backslash R$, it must be the case that $L(f,x)$ and $L(g,x)$ are both 
one-dimensional with the same direction.

Let $L$ be the affine line in $\mathbb{R}^2$ that contains both $L(f,x)$ 
and $L(g,x)$, and let $I$ be the connected component of $L\cap D$ that 
contains $x$. We claim that $I$ must be contained in both $L(f,x)$ and 
$L(g,x)$. Indeed, suppose this is not the case. Then 
there must exist $y\in I\cap \relbd (L(f,x)\cap L(g,x))$.
Therefore $\{y\}$ is a face of either $L(f,x)$ or $L(g,x)$. By Lemma 
\ref{lem:touchingface}, this implies that
$$
	\min\{\dim L(f,y),\dim L(g,y)\}=0,\qquad
	\max\{\dim L(f,y),\dim L(g,y)\}\le 1.
$$
But then $y$ is $(f,g)$-extreme, which entails a contradiction.

As $I$ is contained in both $L(f,x)$ and $L(g,x)$, it follows from
Lemma \ref{lem:touchingface} that $L(f,y)=L(f,x)$ and $L(g,y)=L(g,x)$
for all $y\in I$. Thus $I\cap R=\varnothing$.
\end{proof}

Note that Corollary~\ref{cor:mixedhn} is merely a 
specialization of Corollary \ref{cor:mixedhnnonsm} to the case that $f,g$ 
are smooth, and therefore does not require a separate proof.

\section{Support and projection}
\label{sec:obstruction}

The main aim of this final part of the paper is to clarify the behavior of 
the supports of mixed area measures under projection. As was explained in 
\S\ref{sec:higher}, this presents a basic obstacle to extending the proof 
of Theorem~\ref{thm:mainlower} to more general situations. In the first 
two sections, we will prove the remaining cases of 
Theorem~\ref{thm:mainproj}. In the last section, we present a simple 
example further illuminates some aspects of the structure of 
$(C_1,\ldots,C_{n-1})$-extreme directions.

\subsection{The case $\dim T(K,u)=n-1$ of Theorem \ref{thm:mainproj}}

Let $K\in\mathcal{K}^n$, $u\in S^{n-1}$ with $\dim T(K,u)=n-1$. What is 
special about this situation is that, for sufficiently small 
$\varepsilon>0$, the touching cone $T(K,u)$ bisects $B(u,\varepsilon)$. 
Thus for all $v\in B(u,\varepsilon)$, it must be the case that $T(K,v)\cap 
B(u,\varepsilon)$ lies entirely on one side of the hyperplane that 
contains $T(K,u)$. This property makes it possible to flexibly deform $K$ 
in the direction $T(K,u)^\perp$, as we will presently explain.

Recall that the Minkowski difference of $K,L\in\mathcal{K}^n$ is defined 
as
$$
	K \div L = \{x\in\mathbb{R}^n: x+L \subseteq K\}.
$$
The proof will be based on the following result of Schneider.

\begin{lem}
\label{lem:sch}
Let $K\in\mathcal{K}^n_n$ and $L\in\mathcal{K}^n$, and define
$$
	K_t = \begin{cases}
	K+tL & \text{for }t\ge 0,\\
	K\div (-t)L & \text{for }t<0.
	\end{cases}
$$ 
Suppose that for every $x\in\bd K$, there is $y\in\bd L$ with
$N(K,x)\subseteq N(L,y)$. Then 
$$
	\frac{d}{dt} h_{K_t}(u)\bigg|_{t=0} = h_L(u)
$$
for every $u\in S^{n-1}$. Moreover, there exists $\delta>0$ so that
$$
	\|h_{K_t}-h_K\|_\infty \le \delta|t|
$$
for all sufficiently small $|t|$.
\end{lem}

\begin{proof}
The first statement is given in \cite[Lemma 7.5.4]{Sch14}
for $K,L\in\mathcal{K}^n_n$; however, its proof only uses that   
$K\in\mathcal{K}^n_n$ and extends \emph{verbatim} to any
$L\in\mathcal{K}^n$.
The
second statement is trivial for $t\ge 0$, and is proved in 
\cite[p.\ 425]{Sch14} for $t<0$.
\end{proof}

The assumption of Lemma \ref{lem:sch}, which is formulated in terms of 
normal cones, is automatically implied by the corresponding property of 
touching cones.

\begin{lem}
\label{lem:schtouch}
Let $K,L\in\mathcal{K}^n$, and suppose that $T(K,v)\subseteq T(L,v)$ for
all $v\in S^{n-1}$. Then for every $x\in\bd K$, there is $y\in\bd L$ with
$N(K,x)\subseteq N(L,y)$.
\end{lem}

\begin{proof}
Given any $x\in\bd K$, we can choose $v\in\relint N(K,x)$ and
$y\in \relint F(L,v)$.
Then $N(K,x)=T(K,v)\subseteq T(L,v)\subseteq N(L,y)$.
\end{proof}

Now let $K\in\mathcal{K}^n$, $u\in S^{n-1}$ with $\dim T(K,u)=n-1$, and
$w\in T(K,u)^\perp \cap S^{n-1}$. Choose $\varepsilon>0$ sufficiently 
small that $T(K,u)$ bisects $B(u,\varepsilon)$. Then for every
$v\in B(u,\varepsilon)$, we have that $T(K,v)\cap B(u,\varepsilon)$
is contained either in $H_{w,0}^+$, $H_{w,0}$, or $H_{w,0}^-$. The latter 
sets are precisely the touching cones of the interval $[0,w]$.
This suggests that we aim to apply Lemma \ref{lem:sch} with $L=[0,w]$.

There are two problems with this idea. First, only the intersection of the 
touching cones of $K$ with $B(u,\varepsilon)$, rather than the touching 
cones themselves, are guaranteed to be contained in a touching cone of 
$[0,w]$. Second, Lemma \ref{lem:sch} requires that $K$ has nonempty 
interior, which we have not assumed. The following lemma shows that we can 
modify $K$ in such a way that the conditions of Lemma \ref{lem:sch} are 
satisfied, without changing the part of its boundary with normal 
directions in $B(u,\varepsilon)$.

\begin{lem}
\label{lem:hatk}
Fix $K\in\mathcal{K}^n$ and $u\in S^{n-1}$ with $\dim T(K,u)=n-1$, and
let $w\in T(K,u)^\perp \cap S^{n-1}$. Choose $\varepsilon>0$ sufficiently
small that $T(K,u)$ bisects $B(u,\varepsilon)$.
Then there exists $\tilde K\in\mathcal{K}_n^n$ with the following
properties.
\smallskip
\begin{enumerate}[a.]
\itemsep\smallskipamount
\item $F(\tilde K,v)=F(K,v)$ for all $v\in B(u,\varepsilon)$.
\item $T(\tilde K,v)\subseteq T([0,w],v)$ for all $v\in S^{n-1}$.
\end{enumerate}
\end{lem}

\begin{proof}
Let $\hat B=\mathrm{conv}\{B,su\}$, where $s>1$ is chosen so that
$N(\hat B,su) = \mathbb{R}_+\cl B(u,\varepsilon)$. We define
$\tilde K = K + \hat B - su$. Then clearly
$$
	F(\tilde K,v) = F(K,v) + F(\hat B,v) - su = F(K,v)
$$
for all $v\in B(u,\varepsilon)$, which verifies part $a$. Now note that
\begin{align*}
	T(\tilde K,v) &=
	T(K,v) \cap T(\hat B,v) \\ & =
	\begin{cases}
	T(K,v)\cap \mathbb{R}_+\cl B(u,\varepsilon) & 
	\text{for }v\in B(u,\varepsilon), \\
	\mathbb{R}_+v & \text{for }v\in S^{n-1}\backslash B(u,\varepsilon).	
	\end{cases}
\end{align*}
As $T(K,u)\cap \mathbb{R}_+\cl B(u,\varepsilon)=w^\perp \cap 
\mathbb{R}_+\cl B(u,\varepsilon)$ bisects 
$B(u,\varepsilon)$ and as any two touching cones of 
$K$ are either equal or disjoint, part $b$ follows readily.
\end{proof}

We finally recall the following well known fact, see, e.g.,
\cite[Lemma 2.12]{HR24}.

\begin{lem}
\label{lem:malocal}
Let $K,L,C_1,\ldots,C_{n-2}\in\mathcal{K}^n$ and $A\subseteq S^{n-1}$.
If $F(K,v)=F(L,v)$ for all $v\in A$, then 
$\SM_{K,C_1,\ldots,C_{n-2}}|_A =
\SM_{L,C_1,\ldots,C_{n-2}}|_A$.
\end{lem}

We can now conclude the proof.

\begin{proof}[Proof of the case $\dim T(K,u)=n-1$ of Theorem 
\ref{thm:mainproj}]
Fix $K,C_1,\ldots,C_{n-2}\in\mathcal{K}^n$ and $u\in S^{n-1}$ so that
$\dim T(K,u)=n-1$, and let $w\in T(K,u)^\perp\cap S^{n-1}$. Suppose that 
$u\not\in \supp \SM_{K,C_1,\ldots,C_{n-2}}$.
Then there exists $\varepsilon>0$ so that
$T(K,u)$ bisects $B(u,\varepsilon)$ and
$\SM_{K,C_1,\ldots,C_{n-2}}(B(u,\varepsilon))=0$.
Construct $\tilde K$ as in
Lemma \ref{lem:hatk}.

By Lemma \ref{lem:malocal} and part $a$ of Lemma \ref{lem:hatk}, we have
$\SM_{\tilde K,C_1,\ldots,C_{n-2}}(B(u,\varepsilon))=0$. In particular, if 
$f\in 
C^2(S^{n-1})$ is chosen such that $f(v)>0$ for $v\in B(u,\varepsilon)$ and
$f(v)=0$ for $v\in S^{n-1}\backslash B(u,\varepsilon)$, then
we can write
$$
	\int h_{\tilde K}\,d\SM_{f,C_1,\ldots,C_{n-2}} = 
	\int f \,d\SM_{\tilde K,C_1,\ldots,C_{n-2}} = 0.
$$
Now define $\tilde K_t$ as in Lemma \ref{lem:sch}
with $K\leftarrow \tilde K$ and $L\leftarrow [0,w]$. Then
$$
	\int h_{\tilde K_t}\,d\SM_{f,C_1,\ldots,C_{n-2}} = 
	\int f \,d\SM_{\tilde K_t,C_1,\ldots,C_{n-2}} \ge 0
$$
for all $t$ sufficiently small. In particular, this integral is minimized 
at $t=0$.
By Lemma \ref{lem:schtouch} and part $b$ of Lemma \ref{lem:hatk}, 
the assumption of Lemma \ref{lem:sch} is satisfied. Thus a routine 
application of dominated convergence yields
$$
	\int h_{[0,w]}\,d\SM_{f,C_1,\ldots,C_{n-2}} =
	\frac{d}{dt}\int h_{\tilde K_t}\,d\SM_{f,C_1,\ldots,C_{n-2}}
	\bigg|_{t=0} = 0.
$$
The projection formula (see \S\ref{sec:mvma}) now yields
$$
	\int f\, d\SM_{\proj_{w^\perp}C_1,\ldots,\proj_{w^\perp}C_{n-2}} = 0,
$$
so that $u\not\in \supp 
\SM_{\proj_{w^\perp}C_1,\ldots,\proj_{w^\perp}C_{n-2}}$.
We have therefore proved the contrapositive of the desired statement,
concluding the proof.
\end{proof}

\subsection{The case $1<\dim T(K,u)<n-1$}

We first prove part $b$ of Theorem \ref{thm:mainproj} in the special 
case that $\dim T(K,u)=2$. The construction is then readily 
adapted to any $1<\dim T(K,u)<n-1$ at the end of the proof.

The idea of the proof is to construct 
$K\in\mathcal{K}^n$ with the following properties:
\smallskip
\begin{enumerate}[a.]
\itemsep\medskipamount
\item $\dim T(K,u')=2$ for all $u'$ in a neighborhood of $u$.
\item For every $v\in T(K,u)^\perp$, the union of the
touching cones $T(K,u')$ that are contained in $v^\perp$ 
has Hausdorff dimension at most $n-2$. Thus
$\dim T(\proj_{v^\perp}K,u')=\dim (T(K,u')\cap v^\perp)=1$
on a dense set of $u'$ in a neighborhood of $u$.
\end{enumerate}
\smallskip
Therefore, by Corollary \ref{cor:areasupp}, $u\in\supp 
\SM_{\proj_{v^\perp}K[n-2]}$ for all $v\in T(K,u)^\perp$ but
$u\not\in\supp \SM_{K[n-1]}$, which proves part $b$ of Theorem 
\ref{thm:mainproj} when $\dim T(K,u)=2$.

It is somewhat more convenient to construct the polar body $L=K^\circ$.
The main part of the argument is contained in the following lemma. 
Throughout this section, we denote by $e_1,\ldots,e_n$ the coordinate
basis of $\mathbb{R}^n$.

\begin{lem}
\label{lem:nastybody}
For every $n\ge 4$, there exists $L\in\mathcal{K}^n_{(o)}$
with the following properties.
\begin{enumerate}[a.]
\itemsep\medskipamount
\item $[e_1-e_n,e_1+e_n]$ is a $1$-face of $L$.
\item There is a neighborhood $O\subset \bd L$ of $e_1$ so that 
every $y\in O$ lies in the relative interior of a $1$-face $F_y$ of $L$.
\item For every $(n-1)$-dimensional subspace $E$ of $\mathbb{R}^n$ that 
contains $[e_1-e_n,e_1+e_n]$, the union of all the faces $F_y\subset E$ 
has Hausdorff dimension at most $n-3$.
\end{enumerate}
\end{lem}

\begin{proof}
Let $U$ be the Euclidean unit ball in $\mathbb{R}^{n-1}$ and define
$$
	M = \big\{x\in\mathbb{R}^{n-1}:
	x_1^2+\cdots+x_{n-2}^2+(x_{n-1}-t f(x_1,\ldots,x_{n-2}))^2
	\le 1\big\},
$$
where $f$ is a smooth function to be chosen below such that 
$f(1,0,\ldots,0)=0$.
We choose $t>0$ sufficiently small so that $M$ is strictly convex.

In the following, we will identify $U$ and $M$ with their natural
embedding into $e_n^\perp\subset\mathbb{R}^n$.
We now define the convex body $L\in\mathcal{K}^n_{(o)}$ as
$$
	L = \mathrm{conv}\big(M-e_n,U+e_n\big).
$$
We must check that this body satisfies all the desired properties.

\medskip

\textit{Part $a$.} It is readily seen that $F(U,e_1)=F(M,e_1)=\{e_1\}$.
Therefore $H_{e_1,1}$ is a supporting hyperplane of $L$ and
$F(L,e_1) = L\cap H_{e_1,1} = [e_1-e_n,e_1+e_n]$.

\medskip

\textit{Part $b$.}
Choose $O=B(e_1,\varepsilon)\cap \bd L$ with $\varepsilon<1$. By 
construction, every extreme point of $L$ is contained either in $\bd 
M-e_n$ or $\bd U+e_n$. Thus every $y\in O$ is a non-extreme boundary point 
of $L$, and is therefore in the relative interior of a face $F_y$ of $L$ 
with $\dim F_y\ge 1$. But $F_y\cap (\bd U+e_n)$ and $F_y\cap 
(\bd M-e_n)$ must be singletons, as $U$ and $M$ are strictly convex. Thus 
$\dim F_y=1$ for all $y\in O$.

\medskip

\textit{Part $c$.}
Let $y\in O$ and $w\in N(L,y)\cap S^{n-1}$. By part $b$, the exist points
$a\in \bd M$ and $b\in \bd U$ so that $F_y=[a-e_n,b+e_n]\subseteq 
F(L,w)$. Therefore
$$
	F(M,u)=\{a\} \quad \text{and} \quad
	F(U,u)=\{b\} \quad \text{with} \quad
	u=\frac{\proj_{e_n^\perp}w}{\|\proj_{e_n^\perp}w\|},
$$
where we used that $U$ and $M$ are strictly convex. By the definition of 
$M$, that
$u$ is a normal vector to $M$ at the boundary point $a$ implies that
$$
	u = 
	\frac{\Gamma(a)}{\|\Gamma(a)\|}
	\quad\text{with}\quad
	\Gamma(a) =
	\begin{bmatrix}
	a_1 - (a_{n-1}-tf(a))\, t \partial_1f(a) \\
	\vdots \\
	a_{n-2} - (a_{n-1}-tf(a))\, t \partial_{n-2}f(a) \\
	a_{n-1}-tf(a)
	\end{bmatrix}.
$$
On the other hand, clearly $b=u$.

Let $E$ be an $(n-1)$-dimensional subspace of $\mathbb{R}^n$ that
contains $[e_1-e_n,e_1+e_n]$. Then there exists a vector 
$(v_2,\ldots,v_{n-1})\ne 0$
so that
$$
	E = \{x\in\mathbb{R}^n:
	v_2x_2+\cdots+v_{n-1}x_{n-1}=0\}.
$$
That $F_y\subset E$ requires that $a,b\in E$, 
which yields the system of equations
\begin{align}
\label{eq:nastysystem1}
	&
	a_1^2+\cdots+a_{n-2}^2+(a_{n-1}-t f(a))^2=1,\\
\label{eq:nastysystem2}
	&
	v_2a_2+\cdots+v_{n-1}a_{n-1}=0,\\
\label{eq:nastysystem3}
	&
	(v_2 \partial_2f(a) +
	\cdots +
	v_{n-2} \partial_{n-2}f(a) )
	(a_{n-1}-tf(a)) +
	v_{n-1} f(a) = 0.
\end{align}
It is natural to expect that the set $S$ of solutions 
$a\in\mathbb{R}^{n-1}$ 
of this system of $3$ equations has Hausdorff dimension at most $n-4$.
We choose the function $f$ in the definition of $L$ so that this is indeed 
the case for \emph{every} choice of the vector $v$. An explicit 
construction of such a function may be found
in Appendix \ref{sec:nastyexplicit}, but its precise form is irrelevant
for what follows.

To complete the proof, note that we have shown that every $F_y$
has the form $\big[a-e_n,\frac{\Gamma(a)}{\|\Gamma(a)\|}+e_n\big]$ for 
$a\in S$. As the map $a\mapsto \frac{\Gamma(a)}{\|\Gamma(a)\|}$
is locally Lipschitz,
$$
	\bigcup_{y\in O}F_y = 
	\bigg\{
	\lambda(a-e_n) + (1-\lambda)\bigg(
	\frac{\Gamma(a)}{\|\Gamma(a)\|}+e_n\bigg) :
	a\in S,~\lambda\in[0,1]
	\bigg\}
$$
is the image by a locally Lipschitz function of a set $S\times [0,1]$
of Hausdorff dimension at most $n-3$, and thus has  
dimension at most $n-3$ itself by \cite[Theorem 7.5]{Mat95}.
\end{proof}

We can now prove part $b$ of Theorem \ref{thm:mainproj}
in the case $\dim T(K,u)=2$.

\begin{cor}
\label{cor:nastynasty}
There exists $K\in\mathcal{K}^n_{(o)}$ and $u\in S^{n-1}$ with
$\dim T(K,u)=2$ so that
$u\in\supp \SM_{\proj_{v^\perp}K[n-2]}$ for every
$v\in T(K,u)^\perp$, but $u\not\in\supp \SM_{K[n-1]}$.
\end{cor}

\begin{proof}
Let $K=L^\circ$, where $L$ is the convex body provided by Lemma 
\ref{lem:nastybody} and let $u=e_1$. Then part $a$ of Lemma 
\ref{lem:nastybody} and the duality between faces and touching cones (see 
\S\ref{sec:touchingcones}) yields $T(K,u)=\mathbb{R}_+[e_1-e_n,e_1+e_n]$, 
and thus $\dim T(K,u)=2$. Moreover, by the same argument, part $b$ of 
Lemma \ref{lem:nastybody} shows that there exists $\varepsilon>0$ 
sufficiently small so that $\dim T(K,u')=2$ for all $u'\in 
B(u,\varepsilon)$. Therefore, $u\not\in \supp \SM_{K[n-1]}$ follows 
directly from Corollary \ref{cor:areasupp}.

Finally, by the same duality argument, part $c$ of Lemma 
\ref{lem:nastybody} shows that for every $v\in T(K,u)^\perp$, the union of 
all the touching cones $T(K,u')\subset v^\perp$ with $u'\in 
B(u,\varepsilon)$ has Hausdorff dimension at most $n-2$.
Therefore,
$$
	\dim T(\proj_{v^\perp}K,u') = \dim (T(K,u')\cap v^\perp) =1
$$
for a dense subset of $u'\in B(u,\varepsilon)\cap v^\perp$,
where we used Lemma 
\ref{lem:touchingproj}. In particular, 
$$
	u\in 
	\cl\big\{u'\in S^{n-1}\cap v^\perp:
	\dim T(\proj_{v^\perp}K,u')=1\big\},
$$
and thus $u\in\supp \SM_{\proj_{v^\perp}K[n-2]}$ by
Corollary \ref{cor:areasupp}.
\end{proof}

We now finally extend the argument to the general case.

\begin{proof}[Proof of part $b$ of Theorem \ref{thm:mainproj}]
Let $N\ge 4$ and $1<k<N-1$. We aim to construct bodies
$K,C_1,\ldots,C_{N-2}\in\mathcal{K}^N$ and $u\in S^{N-1}$
with $\dim T(K,u)=k$ so that $u\in \supp 
\SM_{\proj_{v^\perp}C_1,\ldots,\proj_{v^\perp}C_{N-2}}$ 
for every $v\in T(K,u)^\perp$, but
$u\not\in \supp \SM_{K,C_1,\ldots,C_{N-2}}$.

To this end, let $n=N-k+2$, and construct $\bar K\in\mathcal{K}^n$ and 
$\bar u\in S^{n-1}$ as in Corollary~\ref{cor:nastynasty}.
Let $K,u$ be their natural embedding in 
$\mathrm{span}\{e_1,\ldots,e_n\}\subseteq\mathbb{R}^N$, and
$$
	C_1 = \cdots = C_{n-2} = K,
	\qquad
	C_i = [0,e_{i+2}] \quad\text{for}\quad i=n-1,\ldots,N-2.
$$
Note that, by construction,
$$
	T(K,u) = T(\bar K,\bar u) \times \mathbb{R}^{N-n}.
$$
Thus $\dim T(K,u) = k$. On the other hand, by the projection formula 
\eqref{eq:projmixedarea}, the mixed area 
measures $\SM_{K,C_1,\ldots,C_{N-2}}$ and 
$\SM_{\proj_{v^\perp}C_1,\ldots,\proj_{v^\perp}C_{N-2}}$ reduce to the 
$n$-dimensional case in Corollary \ref{cor:nastynasty}, concluding the 
proof.
\end{proof}

\begin{rem}
The bodies $K,C_1,\ldots,C_{N-2}\in\mathcal{K}^N$ in the above proof have
empty interior. However, the construction can be modified as in the proof
of Lemma \ref{lem:hatk} to obtain
$K,C_1,\ldots,C_{N-2}\in\mathcal{K}^N_N$ that yield the same conclusion.
\end{rem}

\subsection{On the continuity of extreme directions}

The construction of the previous section illustrates that the support of 
mixed area measures can be poorly behaved under projection. One is 
therefore led to seek other approaches to Conjecture~\ref{conj:schneider} 
that are based on a direct analysis of convex bodies in $\mathbb{R}^n$. 
The aim of this section is to record an elementary example that 
illustrates a basic challenge that arises even in the most well-behaved 
situations.

A tantalizing setting for understanding Conjecture 
\ref{conj:schneider} is the special case that $h_{C_1},\ldots,h_{C_{n-1}}$ 
are smooth, since in this case there is an analytic description of the 
support that was explained in Remark \ref{rem:smoothmystery}.
The problem remains open even in this setting. This is particularly
surprising since the methods of Hartman and Nirenberg \cite{HN59} provide
powerful information. In particular, \cite[\S3, Lemma 
2]{HN59} implies that for every convex body $K$ with $h_K\in 
C^2(S^{n-1})$, there is a dense open set of $u\in S^{n-1}$ 
for which $\mathop{\mathrm{rank}}(D^2h_K(u)) = \dim (T(K,u)^\perp)$.
Thus the analytic characterization of Remark \ref{rem:smoothmystery}
agrees with Conjecture \ref{conj:schneider} at almost all points.

In the setting of Hartman and Nirenberg (e.g., for the proof of Theorem 
\ref{thm:hn}), the above property suffices to obtain the desired 
conclusion by a continuity argument. A basic problem in the mixed setting, 
however, is that even when all $v\in B(u,\varepsilon)$ are not 
$(C_1,\ldots,C_{n-1})$-extreme, the set $I$ that witnesses the 
failure of extremality (cf.\ Definition \ref{defn:extreme}) can vary with 
$v$ in a discontinuous fashion. 

\begin{example}
\label{ex:cexample}
Since it leads to simpler expressions, we formulate the example 
in terms of mixed Hessian 
measures; the example can be translated to an analogous example
of mixed area measures using the correspondence in \S\ref{sec:mixedhess}.

Let $\Omega = \mathbb{R}\times (-1,1)\subset\mathbb{R}^2$.
The function $h\in\mathrm{Conv}(\Omega)$ defined by
$$
	h(x_1,x_2)=\frac{x_1^2}{1-x_2^2}
$$
is strictly convex whenever $x_1\ne 0$. We can now define 
$f,g\in\mathrm{Conv}(\Omega)$ by setting
$f(x_1,x_2)=h(x_1,x_2)1_{x_1>0}$ 
and $g(x_1,x_2)=h(x_1,x_2)1_{x_1<0}$.
The structure of the sets $L(f,x)$ and
$L(g,x)$ in different regions of $\Omega$ is illustrated in Figure 
\ref{fig:cexample}.

We readily see that every $x\in\Omega$ is not $(f,g)$-extreme, and thus 
$\HH_{f,g}=0$ by Corollary 
\ref{cor:mixedhess}.
However,
the regions $\{x_1<0\}$, 
$\{x_1=0\}$, and $\{x_1>0\}$ all
require a different set $I$ 
to witness non-extremality
(i.e., so that $\dim(\bar L(f_I,x)^\perp)<|I|$).
\end{example}

\begin{figure}
\centering
\begin{tikzpicture}

\fill[color=blue!5!white] (-5,1.5) -- (0,1.5) -- (0,-1.5) -- (-5,-1.5);
\fill[color=red!5!white] (5,1.5) -- (0,1.5) -- (0,-1.5) -- (5,-1.5);

\draw[thick,dashed,color=black!30!white] (-5,1.5) -- (5,1.5);
\draw[thick,dashed,color=black!30!white] (-5,-1.5) -- (5,-1.5);
\draw[thick,color=black!30!white] (-4,1.5) -- (4,1.5);
\draw[thick,color=black!30!white] (-4,-1.5) -- (4,-1.5);

\draw[thick] (0,-1.5) -- (0,1.5);

\draw (2.5,0.15) node {$\scriptstyle\dim L(f,x)=0$};
\draw (2.5,-0.15) node {$\scriptstyle\dim L(g,x)=2$};

\draw (-2.5,0.15) node {$\scriptstyle\dim L(f,x)=2$};
\draw (-2.5,-0.15) node {$\scriptstyle\dim L(g,x)=0$};

\draw (0.5,1.95) node[right] {$\scriptstyle\dim L(f,x)=\dim L(g,x)=1$};
\draw[->] (0.55,1.9) to[out=210,in=25] (0.075,1);

\end{tikzpicture}
\caption{Illustration of Example \ref{ex:cexample}.\label{fig:cexample}}
\end{figure}
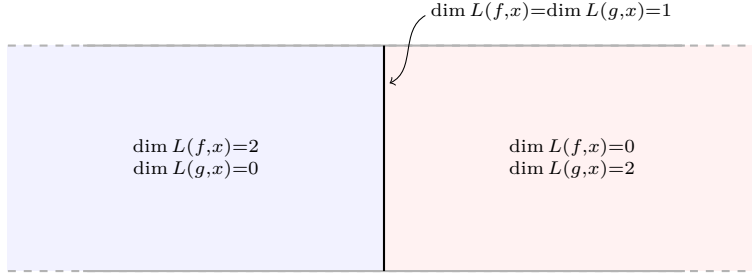

The discontinuous behavior that is illustrated by this example provides a 
basic obstacle to various approaches for the analysis of mixed area 
measures.

\appendix

\section{An explicit construction for Lemma \ref{lem:nastybody}}
\label{sec:nastyexplicit}

The aim of this appendix is to construct a function $f$ with the 
properties required in the proof of Lemma \ref{lem:nastybody}. Any 
sufficiently generic choice of $f$ is expected to achieve the same 
conclusion; we exhibit one explicit example for concreteness.

\begin{lem}
Define the function
$$
	f(x_1,\ldots,x_{n-2}) =
	(x_1-1) \exp\Bigg(
	\sum_{j=2}^{n-2} (x_j^3 + jx_j)
	\Bigg).
$$
Then for any $(v_2,\ldots,v_{n-1})\ne 0$ and $t>0$, the solution set $S$
of the system of 
equations \eqref{eq:nastysystem1}--\eqref{eq:nastysystem3} has
Hausdorff dimension at most $n-4$.
\end{lem}

\begin{proof}
It is clear that the only solution of \eqref{eq:nastysystem1} with
$a_1=1$ is $a=e_1$. We therefore consider in the sequel only
solutions with $a_1\ne 1$. Then \eqref{eq:nastysystem3} can
be simplified to
\begin{equation}
\label{eq:nastysystem3bis}
	\Bigg(\sum_{j=2}^{n-2} v_j (3a_j^2 + j)\Bigg)
	(a_{n-1}-tf(a)) = - v_{n-1}.
\end{equation}
We must now consider two distinct cases.

\medskip

\textbf{Case 1.} If $v_{n-1}\ne 0$, then \eqref{eq:nastysystem2} yields
$$
	a_{n-1} = 
	-\sum_{j=2}^{n-2}\frac{v_ja_j}{v_{n-1}}.
$$
Moreover, as $v_{n-1}\ne 0$, both factors on the left-hand 
side of \eqref{eq:nastysystem3bis} are nonzero. Substituting the 
definition of $f$ and the previous
equation in \eqref{eq:nastysystem3bis}, we obtain
$$
	a_1
	=
	g(a_2,\ldots,a_{n-2})
	=
	1+\frac{1}{t}
	\Bigg(
	\frac{v_{n-1}}{\sum_{j=2}^{n-2} v_j (3a_j^2 + j)}
	-\sum_{j=2}^{n-2}\frac{v_ja_j}{v_{n-1}}
	\Bigg)
	\exp\Bigg(-
	\sum_{j=2}^{n-2} (a_j^3 + ja_j)
	\Bigg).
$$
Finally, combining \eqref{eq:nastysystem1},
\eqref{eq:nastysystem3bis}, and the previous equation yields
$$
	h(a_2,\ldots,a_{n-2})=
	1-g(a_2,\ldots,a_{n-2})^2-a_2^2-\cdots-a_{n-2}^2-\frac{v_{n-1}^2}{
	\big(\sum_{j=2}^{n-2} v_j (3a_j^2 + j)\big)^2}
	=0.
$$
Thus we have shown that
\begin{multline*}
	S\backslash\{e_1\} \subseteq
	\bar S = 
	\Bigg\{
	a\in\mathbb{R}^{n-1}:
	\sum_{j=2}^{n-2} v_j (3a_j^2 + j)\ne 0, ~
	h(a_2,\ldots,a_{n-2})=0,\\
	a_1=g(a_2,\ldots,a_{n-2}),~
	a_{n-1} =-\sum_{j=2}^{n-2}\frac{v_ja_j}{v_{n-1}}	
	\Bigg\}.
\end{multline*}
Now note that $h$ is a function of $n-3$ variables that is real analytic
on $\mathop{\mathrm{dom}}h=\{(a_2,\ldots,a_{n-2}):\sum_{j=2}^{n-2} 
v_j (3a_j^2 + j)\ne 0\}$, and thus the subset of $\mathop{\mathrm{dom}}h$
on which $h$ vanishes has Hausdorff dimension at most $n-4$ \cite{Mit20}.
Since $a_1$ and $a_{n-1}$ are functions of $a_2,\ldots,a_{n-2}$
that are locally Lipschitz on $\mathop{\mathrm{dom}}h$, it follows
that the set $\bar S$ and thus also $S$ has Hausdorff dimension at most 
$n-4$ \cite[Theorem 7.5]{Mat95}. 

\medskip

\textbf{Case 2.} Now suppose that $v_{n-1}=0$. Then 
\eqref{eq:nastysystem3bis} yields $S\backslash\{e_1\}=S_1\cup S_2$ with
\begin{align*}
	S_1&=\big\{a\in S\backslash\{e_1\}: 
	a_{n-1}=tf(a_1,\ldots,a_{n-2})\big\},\\
	S_2&=\Bigg\{a\in S\backslash\{e_1\}:
	\sum_{j=2}^{n-2} v_j (3a_j^2 + j) = 0\Bigg\}.
\end{align*}
We consider $S_1$ and $S_2\backslash S_1$ separately.

To control $S_1$, note that
\eqref{eq:nastysystem1} and \eqref{eq:nastysystem2} imply that
$$
	S_1 \subseteq \Bigg\{ a\in\mathbb{R}^{n-1}:
	\sum_{j=1}^{n-2} a_j^2 =1,~
	\sum_{j=2}^{n-2} v_ja_j=0,~
	a_{n-1}=tf(a_1,\ldots,a_{n-2})\Bigg\}.
$$
In other words, $(a_1,\ldots,a_{n-2})$ lie in the intersection of a sphere 
with a hyperplane, which has dimension $n-4$. As $a_{n-1}$ is a locally
Lipschitz function of $(a_1,\ldots,a_{n-2})$, also $S_1$ has
Hausdorff dimension at most $n-4$ \cite[Theorem 7.5]{Mat95}.

To control $S_2$, let $2\le \ell \le n-2$ so that
$v_\ell\ne 0$. As $\sum_{j=2}^{n-2} v_j (3a_j^2 + j) = 0$ by the 
definition of $S_2$ and as $\sum_{j=2}^{n-2} v_j a_j = 0$ by 
\eqref{eq:nastysystem2}, we obtain
\begin{multline*}
	q(a_2,\ldots,a_{\ell-1},a_{\ell+1},\ldots,a_{n-2}) = \\
	3\sum_{j=2}^{\ell-1} v_j a_j^2  +
	3\sum_{j=\ell+1}^{n-2} v_j a_j^2 +
	3v_\ell \Bigg(
	\sum_{j=2}^{\ell-1} \frac{v_ja_j}{v_\ell}+
	\sum_{j=\ell+1}^{n-2} \frac{v_ja_j}{v_\ell}
	\Bigg)^2
	+ 
	\sum_{j=2}^{n-2} j v_j = 0.
\end{multline*}
Note that $q$ is a quadratic function of $n-4$ variables. Thus its
zero set can have Hausdorff dimension $n-4$ only if $q$ vanishes 
identically. In that case we must have $\sum_{j=2}^{n-2} jv_j =0$ 
and either $v_2=\cdots=v_{\ell-1}=v_{\ell+1}=\cdots=v_{n-2}=0$, or
$v_r=-v_\ell$ for some $2\le r\le n-2$, $r\ne\ell$ and 
$v_j=0$ for $2\le j\le n-2$, $j\ne\ell,r$. Clearly
no such $v$ exists, so the zero set of $q$ has 
Hausdorff dimension at most $n-5$.

Now note that by \eqref{eq:nastysystem2}, we can write
$$
	a_\ell = - \sum_{j=2}^{\ell-1} \frac{v_ja_j}{v_\ell}
	- \sum_{j=\ell+1}^{n-2} \frac{v_ja_j}{v_\ell},
$$
while \eqref{eq:nastysystem1} implies that
$$
	a_{n-1}
	=tf(a_1,\ldots,a_{n-2})\pm\sqrt{1-\sum_{j=1}^{n-2} a_j^2}.
$$
Note in particular that the latter equation implies that
$\sum_{j=1}^{n-2} a_j^2<1$ on the complement of $S_1$. Thus
$S_2\backslash S_1 \subseteq S_+\cup S_-$ with
\begin{multline*}
	S_\pm = \Bigg\{a\in\mathbb{R}^{n-1}:
	\sum_{j=1}^{n-2} a_j^2<1,~
	q(a_2,\ldots,a_{\ell-1},a_{\ell+1},\ldots,a_{n-2}) = 0,\\
	a_\ell = - \sum_{j=2}^{\ell-1} \frac{v_ja_j}{v_\ell}
	- \sum_{j=\ell+1}^{n-2} \frac{v_ja_j}{v_\ell},~	
	a_{n-1}
	=tf(a_1,\ldots,a_{n-2})\pm\sqrt{1-\sum_{j=1}^{n-2} a_j^2}
	\Bigg\}.
\end{multline*}
Since the zero set of $q$ has Hausdorff dimension at most $n-5$ and 
as $a_\ell$ and $a_{n-1}$ are functions of
$a_1,\ldots,a_{\ell-1},a_{\ell+1},\ldots,a_{n-2}$ that are locally 
Lipschitz for $\sum_{j=1}^{n-2} a_j^2<1$, the sets
$S_\pm$ have Hausdorff dimension at most $n-4$ \cite[Theorem 7.5]{Mat95}.
\end{proof}


\subsection*{Acknowledgments}

The authors were supported in part by NSF grants DMS-2054565 and 
DMS-2347954. This paper was completed while R.v.H.\ was a member of the 
Institute for Advanced Study in Princeton, NJ, which is gratefully 
acknowledged for providing a fantastic mathematical environment.

\bibliographystyle{abbrv}
\bibliography{ref}

\begin{thebibliography}{10}

\bibitem{AK52}
R.~D. Anderson and V.~L. Klee, Jr.
\newblock Convex functions and upper semi-continuous collections.
\newblock {\em Duke Math. J.}, 19:349--357, 1952.

\bibitem{BF87}
T.~Bonnesen and W.~Fenchel.
\newblock {\em Theory of convex bodies}.
\newblock BCS Associates, Moscow, ID, 1987.

\bibitem{CH05}
A.~Colesanti and D.~Hug.
\newblock Hessian measures of convex functions and applications to area
  measures.
\newblock {\em J. London Math. Soc. (2)}, 71(1):221--235, 2005.

\bibitem{Fav33}
J.~{Favard}.
\newblock {Sur les corps convexes.}
\newblock {\em {J. Math. Pures Appl. (9)}}, 12:219--282, 1933.

\bibitem{Gut16}
C.~E. Guti\'errez.
\newblock {\em The {M}onge-{A}mp\`ere equation}, volume~89 of {\em Progress in
  Nonlinear Differential Equations and their Applications}.
\newblock Birkh\"auser/Springer, [Cham], second edition, 2016.

\bibitem{HN59}
P.~Hartman and L.~Nirenberg.
\newblock On spherical image maps whose {J}acobians do not change sign.
\newblock {\em Amer. J. Math.}, 81:901--920, 1959.

\bibitem{HMU24}
D.~Hug, F.~Mussnig, and J.~Ulivelli.
\newblock Kubota-type formulas and supports of mixed measures, 2024.
\newblock Preprint arxiv:2401.16371.

\bibitem{HR24}
D.~Hug and P.~A. Reichert.
\newblock The support of mixed area measures involving a new class of convex
  bodies.
\newblock {\em J. Funct. Anal.}, 287(11):Paper No. 110622, 43, 2024.

\bibitem{Kli90}
M.~Kline.
\newblock {\em Mathematical thought from ancient to modern times. {V}ol. 2}.
\newblock The Clarendon Press, Oxford University Press, New York, second
  edition, 1990.

\bibitem{Mat95}
P.~Mattila.
\newblock {\em Geometry of sets and measures in {E}uclidean spaces}, volume~44
  of {\em Cambridge Studies in Advanced Mathematics}.
\newblock Cambridge University Press, Cambridge, 1995.
\newblock Fractals and rectifiability.

\bibitem{Min03}
H.~Minkowski.
\newblock Volumen und {O}berfl\"{a}che.
\newblock {\em Math. Ann.}, 57(4):447--495, 1903.

\bibitem{Min11}
H.~Minkowski.
\newblock Theorie der konvexen {K}\"orper, insbesondere {B}egr\"undung ihres
  {O}berfl\"achenbegriffs.
\newblock In D.~Hilbert, A.~Speiser, and H.~Weyl, editors, {\em Gesammelte
  Abhandlungen von Hermann Minkowski. Zweiter Band}, pages 131--229. B. G.
  Teubner, 1911.

\bibitem{Mit20}
B.~S. Mityagin.
\newblock The zero set of a real analytic function.
\newblock {\em Mat. Zametki}, 107(3):473--475, 2020.

\bibitem{Mun00}
J.~R. Munkres.
\newblock {\em Topology}.
\newblock Prentice Hall, Inc., Upper Saddle River, NJ, second edition, 2000.

\bibitem{Pan85}
A.~A. Panov.
\newblock Some properties of mixed discriminants.
\newblock {\em Mat. Sb. (N.S.)}, 128(170)(3):291--305, 446, 1985.

\bibitem{Pog73}
A.~V. Pogorelov.
\newblock {\em Extrinsic geometry of convex surfaces}, volume Vol. 35 of {\em
  Translations of Mathematical Monographs}.
\newblock American Mathematical Society, Providence, RI, 1973.
\newblock Translated from the Russian by Israel Program for Scientific
  Translations.

\bibitem{Roc70}
R.~T. Rockafellar.
\newblock {\em Convex analysis}, volume No. 28 of {\em Princeton Mathematical
  Series}.
\newblock Princeton University Press, Princeton, NJ, 1970.

\bibitem{Sch75}
R.~Schneider.
\newblock Kinematische {B}er\"{u}hrma\ss e f\"{u}r konvexe {K}\"{o}rper und
  {I}ntegralrelationen f\"{u}r {O}berfl\"{a}chenma\ss e.
\newblock {\em Math. Ann.}, 218(3):253--267, 1975.

\bibitem{Sch85}
R.~Schneider.
\newblock On the {A}leksandrov-{F}enchel inequality.
\newblock In {\em Discrete geometry and convexity}, volume 440 of {\em Ann. New
  York Acad. Sci.}, pages 132--141. New York Acad. Sci., 1985.

\bibitem{Sch88}
R.~Schneider.
\newblock On the {A}leksandrov-{F}enchel inequality involving zonoids.
\newblock {\em Geom. Dedicata}, 27(1):113--126, 1988.

\bibitem{Sch14}
R.~Schneider.
\newblock {\em Convex bodies: the {B}runn-{M}inkowski theory}.
\newblock Cambridge University Press, expanded edition, 2014.

\bibitem{SvH19}
Y.~Shenfeld and R.~van Handel.
\newblock Mixed volumes and the {B}ochner method.
\newblock {\em Proc. Amer. Math. Soc.}, 147(12):5385--5402, 2019.

\bibitem{SvH23}
Y.~Shenfeld and R.~van Handel.
\newblock The extremals of the {A}lexandrov-{F}enchel inequality for convex
  polytopes.
\newblock {\em Acta Math.}, 231(1):89--204, 2023.

\bibitem{TW02}
N.~S. Trudinger and X.-J. Wang.
\newblock Hessian measures. {III}.
\newblock {\em J. Funct. Anal.}, 193(1):1--23, 2002.

\bibitem{Wei12}
S.~Weis.
\newblock A note on touching cones and faces.
\newblock {\em J. Convex Anal.}, 19(2):323--353, 2012.

\end{thebibliography}

\end{document}